\documentclass[10pt,reqno]{amsart}

\usepackage{amsmath}
\usepackage{amsthm}
\usepackage{amssymb}
\usepackage{amsfonts}
\usepackage{amsxtra}
\usepackage{fullpage}
\usepackage{amscd}
\usepackage{mathrsfs}
\usepackage[all]{xypic}
\usepackage{enumerate}
\usepackage[noautoscale,enableskew]{youngtab}
\usepackage{graphicx}
\usepackage[mathscr]{eucal}
\usepackage{verbatim}

\let\nc\newcommand
\let\renc\renewcommand

\theoremstyle{plain}
\newtheorem*{thm}{Theorem}
\newtheorem*{prop}{Proposition}
\newtheorem*{cor}{Corollary}
\newtheorem*{lem}{Lemma}

\theoremstyle{definition}
\newtheorem*{defn}{Definition}
\newtheorem*{example}{Example}
\newtheorem*{remark}{Remark}
\newtheorem*{rem}{Remark}

\nc{\bdm}{\begin{displaymath}}
\nc{\edm}{\end{displaymath}}
\nc{\bthm}{\begin{thm}}
\nc{\ethm}{\end{thm}}
\nc{\blem}{\begin{lem}}
\nc{\elem}{\end{lem}}
\nc{\bcor}{\begin{cor}}
\nc{\ecor}{\end{cor}}
\nc{\bprop}{\begin{prop}}
\nc{\eprop}{\end{prop}}
\nc{\bdef}{\begin{defn}}
\nc{\eddef}{\end{defn}}

\makeatletter
\renewcommand{\subsection}{\@startsection{subsection}{2}{0pt}{-3ex
plus -1ex minus -0.2ex}{-2mm plus -0pt minus
-2pt}{\normalfont\bfseries}} \makeatother

\newcommand{\Lmod}[1]{#1\text{-}{\mathsf{mod}}}
\newcommand{\grLmod}[1]{#1\text{-}{\mathsf{grmod}}}

\newcommand{\blambda}{\boldsymbol{\lambda}}
\newcommand{\bmu}{\boldsymbol{\mu}}

\DeclareMathOperator{\im}{\mathrm{Im}}
\DeclareMathOperator{\supp}{\mathrm{Supp}}

\DeclareMathOperator{\End}{\mathrm{End}}

\DeclareMathOperator{\gr}{\mathrm{gr}}

\newcommand{\beq}{\begin{equation}\label}
\newcommand{\eeq}{\end{equation}}

\DeclareMathOperator{\Spec}{\mathrm{Spec}}

\DeclareMathOperator{\Hom}{\mathrm{Hom}}

\nc{\Z}{\mathbb{Z}}
\newcommand{\N}{\mathbb{N}}

\newcommand{\C}{\mathbb{C}}

\newcommand{\Fun}{\mbox{\mathrm{Fun}}\,}

\nc{\rank}{\textrm{rank} \,}
\nc{\ds}{\dots}

\let\mc\mathcal
\let\mf\mathfrak

\nc{\HW}{\bar{H}_{\mathbf{c}}(W)}
\nc{\HK}{\bar{H}_{\mathbf{c}}(K)}
\nc{\HtK}{\widetilde{H}_{\mathbf{c}}(K)}
\nc{\CMW}{\textsf{CM}_{\mbf{c}}(W)}
\nc{\CMK}{\textsf{CM}_{\mbf{c}}(K)}
\nc{\mbf}{\mathbf}
\nc{\LK}{\textsf{Irr}(K)}
\nc{\LW}{\textsf{Irr}(W)}
\nc{\Res}{\mathsf{Res} \, }
\nc{\Ind}{\mathsf{Ind} \, }

\nc{\cont}{\textrm{cont}}

\nc{\eWb}{\mathbf{e}_{W_b}}
\nc{\eW}{\mathbf{e}_{W}}
\nc{\msf}{\mathsf}
\nc{\Ui}{\mc{U}_{i,+}}
\nc{\Uone}{\mc{U}_{1,+}}
\nc{\Utwo}{\mc{U}_{2,+}}

\nc{\minusone}{-1}
\nc{\minustwo}{-2}
\nc{\Mod}{\mathrm{Mod} \,}
\nc{\ms}{\mathscr}
\nc{\Frac}{\mathrm{Frac} \,}
\nc{\ra}{\rightarrow}
\nc{\hra}{\hookrightarrow}
\nc{\lab}{\label}
\renc{\O}{\mc{O}}
\nc{\Tan}{\mc{T}}
\nc{\ul}{\underline}
\nc{\s}{\mathfrak{S}}
\nc{\g}{\mf{g}}
\nc{\pa}{\partial}
\nc{\tit}{\textit}
\nc{\Maxspec}{\mathrm{Maxspec} \, }
\nc{\gldim}{\mathrm{gl.dim}}
\nc{\rkm}{\mathrm{rk} \, (\mf{m})}
\nc{\sm}{\mathrm{sm}}
\nc{\PD}{\mathbb{PD}}
\nc{\hilb}{\textrm{Hilb}}
\nc{\<}{\langle}
\renc{\>}{\rangle}
\nc{\T}{\mathbb{T}}
\nc{\X}{\mathbb{X}}
\nc{\W}{\mathscr{W}}
\nc{\kt}{\mbf{k}}
\nc{\ko}{\mbf{k}(0)}
\nc{\Ok}{\mc{O}_G \boxtimes \kt_X}
\nc{\Oko}{\mc{O}_G \boxtimes \ko_X}
\nc{\OYk}{\mc{O}_Y \boxtimes \kt_X}
\nc{\id}{\msf{id}}
\nc{\A}{\mathbb{A}}
\nc{\Grel}{\mc{Grel}}
\nc{\Grat}{\mc{Grat}}
\nc{\Squo}[1]{\A^{(#1)}}
\nc{\twist}{\mathrm{twist}}
\nc{\Cd}{\mc{C}}
\nc{\Span}{\mathrm{Span}}
\nc{\Grass}{\mathrm{Gr}}
\nc{\Fr}{\mathrm{Fr}}
\nc{\pco}[1]{k[V]^{p\mathrm{co} #1}}
\nc{\Irr}{\mathsf{Irr} }
\renc{\o}{\otimes}
\renc{\gr}{\mathsf{gr}}
\nc{\U}{\mathsf{U}}
\nc{\algD}{\mf{D}}

\nc{\hr}{\mf{h}_{\textrm{reg}}}
\nc{\D}{\mathscr{D}}
\nc{\PIdeg}{\mathrm{P.I.-degree}}
\nc{\ch}{\mathrm{ch}}
\nc{\ev}{\mathsf{ev}}
\nc{\Stab}{\mathrm{Stab}}
\nc{\Der}{\mathrm{Der}}
\nc{\rightsim}{\stackrel{\sim}{\longrightarrow}}
\nc{\HZ}{H_{\mbf{h},\Z}(\Z_m)}
\nc{\sing}{\mathrm{sing}}
\nc{\dd}{\mathscr{D}}
\nc{\GKdim}{\mathrm{G.K. dim}}
\nc{\PIdegree}{\mathrm{P.I. degree}}
\renc{\H}{\mathsf{H}}
\nc{\rH}{\overline{\mathsf{H}}}
\renc{\Fun}{\mathrm{Fun}}
\nc{\bc}{\mathbf{c}}
\nc{\vc}{\underline{\mathbf{c}}}
\nc{\ba}{\mathbf{a}}

\begin{document}

\title{On the smoothness of centres of rational Cherednik algebras in positive characteristic}

\author{Gwyn Bellamy}

\address{School of Mathematics, Room 2.233, Alan Turing Building, University of Manchester, Oxford Road, Manchester, M13 9PL}
\email{gwyn.bellamy@manchester.ac.uk}

\author{Maurizio Martino}

\address{Mathematisches Institut, Endenicher Allee 60, 53115 Bonn, Germany}
\email{mmartino@math.uni-bonn.de}

\begin{abstract}
In this article we study rational Cherednik algebras at $t = 1$ in positive characteristic. We study a finite dimensional quotient of the rational Cherednik algebra called the restricted rational Cherednik algebra. When the corresponding pseudo-reflection group belongs to the infinite series $G(m,d,n)$, we describe explicitly the block decomposition of the restricted algebra. We also classify all pseudo-reflection groups for which the centre of the corresponding rational Cherednik algebra is regular for generic values of the deformation parameter.
\end{abstract}

\maketitle

\centerline{\it Dedicated, with admiration and thanks, to Ken Brown}
\centerline{\it and Toby Stafford on their $60$th birthdays}

\section{Introduction}

\subsection{} Rational Cherednik algebras were introduced by Etingof and Ginzburg in $2002$. Since their introduction they have been extensively studied and have been shown to be related to many other branches of mathematics such as integrable systems, symplectic algebraic geometry and algebraic combinatorics. In this article we continue the study, initiated in \cite{BFG}, \cite{BrownChangtong}, \cite{BalChen1} and \cite{BalChen2}, of these algebras at $t = 1$ and over a field of positive characteristic. We focus on the representation theoretic aspects of the story. In particular, we examine the block structure of certain finite-dimensional quotient algebras called restricted rational Cherednik algebras. We also look at the question of when the centre of the rational Cherednik algebra is smooth. Analogous problems have already been solved for rational Cherednik algebras at $t = 0$ in characteristic zero, see \cite{EG}, \cite{Baby}, \cite{Singular},  \cite{MarsdenWeinsteinStratification}, \cite{Mo}, \cite{CMpartitions}, and our results are very similar in nature. The methods we develop, however, are new, and in fact can be used to reprove many of the characteristic zero results.

\subsection{}\label{ss:maintheorem1} Let us review our results. Further details can be found in the main body of the paper. Let $W$ be a pseudo-reflection group. Let $k$ be an algebraically closed field of characteristic $p$ with $p \nmid |W|$. Let $V$ denote the reflection representation of $W$ over $k$. Let $\mc{S}(W)$ denote the set of reflections in $W$ and let $\bc : \mc{S}(W) \to k$ be a $W$-invariant function. To this data we can attach a $k$-algebra $\H_{\bc}(W)$ called the {\it rational Cherednik algebra}.

Let $V^{(1)}$ denote the Frobenius twist of $V$. Let $Z_{\bc}(W)$ denote the centre of $\H_{\bc}(W)$. There is an injective algebra homomorphism \[ k[V^{(1)}]^W \otimes k[(V^*)^{(1)}]^W \hookrightarrow Z_{\bc}(W). \] Factoring out by the unique graded, maximal ideal in this central subalgebra, one gets a finite dimensional, graded quotient of the rational Cherednik algebra. This factor algebra is called the restricted rational Cherednik algebra and is denoted $\rH_{\bc}(W)$. Simple modules for this finite dimensional algebra are in natural bijection with the simple modules of the group $W$. Thus, the blocks of $\rH_{\bc}(W)$ give us a partition of the set $\Irr W$. Our first main result is an explicit combinatorial description of this block partition when $W$ belongs to the infinite series $G(m,1,n)$. In this case, the irreducible representations of $W$ are naturally labeled by $\mathcal{P}(m,n)$, the set of $m$-multipartitions of $n$. For each $\blambda \in \mc{P}(m,n)$, let $L_{\bc}(\blambda)$ denote the corresponding simple $\rH_{\bc}(W)$-module. For the definitions in the following statement and the proof of the following theorem, see \ref{blocks}.

\begin{thm}
Let $\blambda, \bmu \in \mathcal{P}(m,n)$. Let $a = (0 , H_1, H_1 + H_2, \dots , H_1 + \dots + H_{m-1})$. Then $L_{\bc}(\blambda)$ and $L_{\bc}(\bmu)$ belong to the same block of $\rH_{\bc}(W)$ if and only if
$$
\sum_{i=0}^{m-1} x^{({a}_i^p - a_i)} \Res_{\lambda^i}(x^{-(\kappa^p - \kappa)}) = \sum_{i=0}^{m-1} x^{({a}_i^p - a_i)} \Res_{\mu^i}(x^{-(\kappa^p - \kappa)}).
$$
\end{thm}

Using this description of the blocks of the restricted rational Cherednik algebra $\rH_{\bc}(G(m,1,n))$, some Clifford theory also allows one to describe the blocks of the restricted rational Cherednik algebra $\rH_{\bc}(G(m,d,n))$, see \ref{ss:G(m,d,n)blocks}. Here $G(m,d,n)$ denotes (in the Shepard-Todd classification) the normal subgroup of $G(m,1,n)$ where we impose the restriction that $d \mid m$ and either $n > 2$ or $n=2$ and $d$ is odd.

\subsection{}\label{smooth+blocks} Our remaining results concern the smoothness of $Z_{\bc}(W)$. In section \ref{sec:smoothness} we relate the smoothness of $Z_{\bc}(W)$ to the representation theory of $\H_{\bc}(W)$. Taking as our starting point the fact that the smooth and Azumaya loci of $Z_{\bc}(W)$ are equal, we use the restriction functors of \cite{BE} to establish the following.

\begin{thm} The following are equivalent:
\begin{enumerate}
\item $Z_{\bc}(W)$ is smooth;
\item the blocks of $\rH_{\bc'}(W')$ are singletons for all parabolic subgroups $W' \subseteq W$.
\end{enumerate} Here, $\bc'$ denotes the restriction of $\bc$ to $\mc{S}(W) \cap W'$.
\end{thm}

In the case that $W=G(m,1,n)$, we can apply the theorem together with the description of blocks from Theorem \ref{ss:maintheorem1} to determine for which parameters $\bc$ the centre $Z_{\bc}(W)$ is smooth.

\begin{cor}\label{cor:GMoneNsmooth}
The centre of $\H_{\bc}(G(m,1,n))$ is smooth if and only if $\bc$ does not lie on the finitely many hyperplanes in $\mc{C}$ defined by $$\kappa \in \mathbb{F}_p\quad \mathrm{and}\quad
a_i - a_j \pm C \kappa \in \mathbb{F}_p, \quad \forall \ 0 \le i \neq j \le m-1,\ 0 \leq C \leq n-1.
$$
\end{cor}

We should note that an analogous version of the theorem and its corollary is true for rational Cherednik algebras at $t=0$ over the complex numbers. Using \cite[Theorem 5.5]{Mo}, this gives an alternative way to describe the parameter values where the centre is smooth, cf. \cite[Lemma 4.3]{GordonQuiver}. The theorem also clarifies the relationship between restricted Cherednik algebras and the smoothness of the centres $Z_{\bc}(W)$.

\subsection{}\label{genericsmooth} The corollary shows that the centre of $\H_{\bc}(G(m,1,n))$ is a smooth algebra for generic values of the parameter $\bc$. One can ask more generally: for which pseudo-reflection groups $W$ is $Z_{\bc}(W)$ smooth for generic values of the parameter $\bc$? Our final result answers this question.

\begin{thm}\label{thm:classify}
The centre of the rational Cherednik algebra $\H_{\bc}(W)$ is never smooth if $W$ is not isomorphic to $G(m,1,n)$ or $G_4$. If $W$ is isomorphic to $G(m,1,n)$ for some $m$ and $n$ or to the exceptional group $G_4$, then the centre of $\H_{\bc}(W)$ is smooth for generic values of $\bc$.
\end{thm}

Our approach to proving Theorem \ref{thm:classify} follows the same path as in \cite{Singular} - we calculate the Poincar\'e polynomial of the graded $\rH_{\bc}(W)$-modules $L_{\bc}(\blambda)$ under the assumption that the dimension of these modules is maximal, namely $\dim_k L_{\bc}(\blambda) = p^n |W|$. This leads to a contradiction for many choices of $\blambda$. However, the naive argument of \tit{loc. cit.} is not sufficient in our case to calculate this Poincar\'e polynomial. Instead, we show that $L_{\bc}(\blambda)$ can be deformed (flatly) to a graded $\H_0(W) = \dd(V) \rtimes W$-module. Using a result of Cartier's on $\dd$-modules with zero $p$-curvature, we study the graded $W$-character of this module at $\bc = 0$. 

\begin{rem} 
Theorem \ref{thm:classify} and Corollary \ref{ss:maintheorem1} provide a complete answer to \cite[Question D]{BrownChangtong} for rational Cherednik algebras.
\end{rem}

\subsection{} The paper is structured as follows. In section 2 we introduce notation and recall some facts about pseudo-reflection groups. In section 3 we define rational Cherednik algebras and state their main properties. The Dunkl-Opdam operators are introduced in Section 4 and are used to prove Theorem \ref{ss:maintheorem1}. Section 5 is devoted to parabolic restriction and induction and their application in the proof of Theorem \ref{smooth+blocks}. Finally, in section 6 we establish certain properties of $\dd$-modules in characteristic $p$ and use these to prove Theorem \ref{genericsmooth}.    

\subsection{Acknowledgements}

The first author is supported by the EPSRC grant EP-H028153. The second author was supported by the SFB/TR 45 ``Periods, Moduli
Spaces and Arithmetic of Algebraic Varieties" of the DFG (German Research Foundation). The authors would like to thank Ulrich Thiel for suggesting that the Euler element should distinguish the blocks for $G_4$ at generic parameters, and for showing us an early version of \cite{Thiel}. 

\section{Basics}

\subsection{Definitions and notation}\label{subsection:defns}

Let $k$ be an algebraically closed field of characteristic $p > 0$. Let $V$ be a $k$-vector space of finite dimension and $W$ a finite group acting linearly on $V$. An element $s \in W$ is called a pseudo-reflection if the fixed space of $s$ has co-dimension one. Let $\mathcal{S}(W)$ denote the set of all pseudo-reflections in $W$. Then $W$ is said to be a \tit{pseudo-reflection group} if $W = \langle \mathcal{S}(W) \rangle$, a good reference on the theory of pseudo-reflection groups is \cite{LehrerTaylor}. We assume throughout that $\mathrm{char} (k)$ does \textbf{not} divide $|W|$. This assumption on the characteristic of $k$ implies that there are no transvections in $\mathcal{S}(W)$. One can also check from the classification of pseudo-reflection groups, as recalled in \cite{KemperMalle}, that this assumption implies that $W$ is the reduction mod $p$ of a complex reflection group.

\subsection{Frobenius twists and group actions} Let $V$ be a finite dimensional vector space over
$k$. Let $W$ be a finite group acting linearly on $V$ and assume that $p$ does not divide $|W|$. The Frobenius morphism $\Fr
: k[V] \ra k[V]$ is the ring homomorphism $f \mapsto f^p$. Denote by $k[V^{(1)}]$ the image of $\Fr$ but with twisted linear
structure $z \star f = z^p f$ for all $f \in k[V^{(1)}]$ and $z \in k$. Then $\Fr : k[V] \ra k[V^{(1)}]$ is a $k$-linear
isomorphism. Note that $k[V]$ is a finite free $k[V^{(1)}]$-module of rank $(\dim V)^p$. It is easy to check that $\Fr$ is a $W$-equivariant ring homomorphism, so we have an isomorphism $\Fr^W : k[V]^W \ra k[V^{(1)}]^W$. In particular, the $p$-th powers of a generating set for $k[V]^W$ form a generating set for $k[V^{(1)}]^W$.


\subsection{Representations of pseudo-reflection groups}\label{RepsReduction} Let us retain the notation from above. Let $K$ be a finite field extension of $\mathbb{Q}$ containing all the $|W|$th roots of unity, let $A$ be the localisation of the ring of integers of $K$ at the prime ideal generated by $p \in \mathbb{Z}$ and let $L$ be the residue field of $A$. Note that $L$ is a finite field of order a power of $p$. With this setup we can define a decomposition map on characters of irreducible representations as follows, see \cite[$\S$ 7]{GeckPfeiffer}. Let $M$ be an irreducible $K W$-module. We can choose an $A$-lattice $M_A$ in $M$ so that the action of $W$ on $M_A$ has structure constants in $A$. Let $M_L$ denote the reduction of $M_A$ to $L$, which is naturally a $L W$-module. The decomposition map is the assignment $\chi \mapsto {\chi}_L$, where $\chi, \chi_L$ denote the characters of $M$ and $M_L$, respectively. By \cite[Corollary 17.2]{CurtisReiner} and Tits' deformation theorem, \cite[Theorem 7.4.6]{GeckPfeiffer}, we have the following.

\begin{thm}
Both $K W$ and $L W$ are split algebras and the decomposition map defines a bijection between the irreducible characters of $K W$ and $L W$.
\end{thm}

In particular, it follows that the irreducible characters of $k W$ are given by reducing the irreducible characters of $\C W$ to $k$. 

\subsection{p-coinvariant rings} Let $W$ be a pseudo-reflection group and let $K(W)$ denote the Grothendieck group of finite dimensional $W$-modules. Let $\Irr W$ be a complete set of isomorphism classes of simple $W$-modules. We denote by $\Irr_{\Z} W$ a complete set of isomorphism classes of graded, simple $W$-modules. For a graded $W$-module $M$, we write $\ch_{t,W}(M) \in K(W)[t,t^{-1}]$ for its graded character. The shift $M[i]$ of $M$ is the graded $W$-module such that $M[i]_j = M_{j - i}$.

We endow the algebra $k[V]$ with its usual $\N$-grading. Let $d_1, \ds, d_n$ be the degrees of a set of fundamental homogeneous algebraically independent generators of $k[V]^W$. The coinvariant ring of $W$ is $k[V]^{\mathrm{co} W}$, and for each $\lambda \in \Irr W$, we denote by $f_\lambda(t)$ the corresponding fake polynomial, defined by
$$
\ch_{t,W} (k[V]^{\mathrm{co} W}) = \sum_{\lambda \in \Irr W} f_{\lambda}(t) \cdot [\lambda].
$$

\bdef
Let $k[V^{(1)}]^W_+$ denote the invariant polynomials with zero constant term. The \tit{$p$-coinvariant} ring is
defined to be the finite dimensional graded algebra $\pco{W} := k[V] / \< k[V^{(1)}]^W_+ \>$.
\eddef

\begin{lem}\label{lem:pcochar}
Keep notation as above. Then there is an isomorphism of graded $W$-modules
\begin{align}\label{eq:pcoiso}
\pco{W} & \simeq (k[V] / \< k[V]^W_+ \>) \o (k[V]^W / \< k[V^{(1)}]^W_+ \>) \\
 & \simeq (k[V] / \< k[V^{(1)}]_+ \>) \o (k[V^{(1)}] / \< k[V^{(1)}]^W_+ \>).
\end{align}
\end{lem}

\begin{proof}
Consider the inclusions of algebras
\[ \xymatrix{
& k[V]^W \ar[dr] & \\  k[V^{(1)}]^W \ar[ur] \ar[dr] & & k[V] \\ & k[V^{(1)}] \ar[ur] &  }\] Choose a basis $x_1, \dots ,x_n$
of $V^*$, homogeneous generators $f_1, \dots ,f_n$ of $k[V]^W$ and a free homogeneous basis $b_1, \dots ,b_{|W|}$ of $k[V]$
as a $k[V]^W$-module. Then there are free bases for $k[V]$ as a $k[V^{(1)}]^W$-module given by: \[ \{ f^{\alpha} b_j \mid 0
\leq \alpha_i \leq p-1, 1 \leq j \leq |W|\}, \] and \[ \{ x^{\alpha} b_j^p \mid 0 \leq \alpha_i \leq p-1, 1 \leq j \leq
|W|\}. \] Taking the images of these bases in $\pco{W}$ yields the lemma.
\end{proof}

\begin{rem}
The isomorphism (\ref{eq:pcoiso}) implies that
$$
\sum_{i \in \Z} [\pco{W}_i : \lambda]t^i = f_\lambda(t) \cdot \prod_{i = 1}^n \frac{1 - t^{p d_i}}{1 - t^{d_i}}
$$
for all $\lambda \in \Irr W$.
\end{rem}

\section{Rational Cherednik algebras}

\subsection{}\label{sec:defnCherednik} 
For $s \in \mathcal{S}(W)$, fix $\alpha_s \in V^*$ to be a basis of the one dimensional space $\im (s - 1)|_{V^*}$ and $\alpha_s^{\vee} \in V$ a basis of the one dimensional space $\im (s - 1)|_{V}$, normalised so that $\alpha_s(\alpha_s^\vee) = 1$. Let $\mc{C}$ denote the space of $W$-equivariant functions $\mathcal{S}(W) \rightarrow k$ and choose $\bc \in \mc{C}$ and $t\in k$. The \textit{rational Cherednik algebra}, $\H_{t, \bc}(W)$, as introduced by Etingof and Ginzburg \cite[page 250]{EG}, is the quotient of the skew group algebra of the tensor algebra, $T(V \oplus V^*) \rtimes W$, by the ideal generated by the relations
\begin{equation}\label{eq:rel}
[x,x'] = 0, \qquad [y,y'] = 0, \qquad [y,x] = t x(y) - \sum_{s \in \mathcal{S}} \bc(s) \alpha_s(y) x(\alpha_s^\vee) s,
\end{equation}
for all $x,x' \in V^*$ and $y,y' \in V$. We define a filtration $\mc{F}_{\bullet}$ on $\H_{t, \bf c}(W)$ via $\mc{F}_0 = k W$, $\mc{F}_1 = kW \otimes (V \oplus V^*)$ and $\mc{F}_i = \mc{F}_1^i$ for $i > 1$. By \cite[Theorem 1.3]{EG}, there is an isomorphism of algebras
\begin{equation}\label{PBW}
\gr_{\mc{F}} \H_{t, \mathbf{c}}(W) \cong S(V \oplus V^*) \rtimes W.
\end{equation}
As a consequence, there is a vector space isomorphism
\begin{equation}\label{eq:PBW}
\H_{t, \mbf{c}}(W) \cong k [V] \otimes k W \otimes k [V^*].
\end{equation}
There is also a $\Z$-grading on $\H_{\bc}(W)$ given by setting $\mathrm{deg}(W) =0$, $\mathrm{deg}(V)=-1$ and $\mathrm{deg}(V^*)=1$. Throughout this article we assume that $t \neq 0$. Therefore, without loss of generality $t \equiv 1$ and we write $\H_{\bf c}(W)$ for $\H_{1, \mathbf{c}}(W)$. Let $x_1, \ds, x_n$ be a basis of $V^*$ and $y_1, \ds, y_n \in V$ the dual basis. Define the Euler element in $\H_{\bf c}(W)$ to be
$$
\mathbf{h} = \sum_{i = 1}^n x_i y_i - \sum_{s \in S} \frac{c_s}{1 - \lambda_s} s .
$$
One can easily check that $[\mathbf{h},x] = x$, $[\mathbf{h},y] = -y$ and $[\mathbf{h},w] = 0$ for all $x \in V^*$, $y \in V$ and $w \in W$. Therefore the element $\mathbf{h}^p - \mathbf{h}$ belongs to the centre of $\H_{\bc}(W)$.

\subsection{}\label{Cherednik_props}

Below we summarize fundamental the properties of $\H_{\bc}(W)$. Proofs of all these statements can be found in \cite{BrownChangtong}.

\begin{prop}
Let $\H := \H_{\bc}(W)$ be a rational Cherednik algebra associated to $(V,W)$.
\begin{enumerate}
\item The P.I. degree of $\H$ equals $p^n |W|$.
\item The centre $Z_{\bc}(W)$ of $\H$ is an affine domain and the algebra $\H$ is a finite module over its centre.
\item The smooth locus of $Z_{\bc}(W)$ equals the Azumaya locus of $\H$.
\item The commutative subalgebras $k[V^{(1)}]^W$ and $k[(V^*)^{(1)}]^W$ of $\H$ are central.
\end{enumerate}
\end{prop}

\subsection{The restricted rational Cherednik algebra}\label{rrca}

Let $L$ be a simple, graded $\H_{\bc}(W)$-module. The centre of $\H_{\bc}(W)$ acts as a scalar on $L$ and the grading on $L$ implies that $(k[V^{(1)}]^W \o k[(V^*)^{(1)}]^W)_+$ annihilates $L$. Therefore to study these simple, graded modules it suffices to consider a certain graded, finite dimensional quotient of $\H_{\bc}(W)$. 

\begin{defn}
The \tit{restricted rational Cherednik algebra} $\rH_{\mbf{c}}(W)$ is the finite dimensional quotient of $\H_{\bc}(W)$ by the central ideal generated by $(k[V^{(1)}]^W \o k[(V^*)^{(1)}]^W)_+$.
\end{defn}

The algebra $\rH_{\bf c}(W)$ is $\Z$-graded and has dimension $p^{2n } |W|^3$. The PBW property (\ref{eq:PBW}) implies that 
$$
\rH_{\bf c}(W) \cong k [V]^{p \mathrm{co} W} \otimes k W \otimes k [V^*]^{p \mathrm{co} W}
$$
as vector spaces. Since $\pco{W}$ is a complete intersection, \cite[Corollary 21.19]{Eisenbud} implies that it is Gorenstein and thus equipped with a non-degenerate bilinear form. Then the proof of \cite[Theorem 3.6]{BGS} shows that the algebra $\rH_{\mbf{c}}(W)$ is symmetric. 

\begin{defn}
Let $\lambda \in \Irr_{\Z} (W)$. The \tit{baby Verma module} $\bar{\Delta} (\lambda)$ associated to $\lambda$ is the induced module
$$
\bar{\Delta} (\lambda) = \rH_{\mbf{c}}(W) \o_{A} \lambda,
$$
where $A = k[V^*]^{\mathrm{pco} W} \rtimes W$ and the natural action of $W$ on $\lambda$ extends to $A$ by making $V$ act by zero.
\end{defn}

As was done in \cite{Baby},  we can apply the theory developed in \cite{HN} to the category $\grLmod{\rH_{\mbf{c}}(W)}$ of finite-dimensional, graded, left $\rH_{\mbf{c}}(W)$-modules. The forgetful functor $\grLmod{\rH_{\mbf{c}}(W)} \ra \Lmod{\rH_{\mbf{c}}(W)}$ is denoted $F$.

\begin{prop}\label{prop:listofproperties}
Let $\lambda,\mu \in \Irr_{\Z} W$.
\begin{enumerate}
\item The baby Verma module $\bar{\Delta} (\lambda)$ has a simple head $L(\lambda)$.
\item $\bar{\Delta} (\lambda)$ is isomorphic to $\bar{\Delta} (\mu)$ if and only if $\lambda = \mu$ in $\Irr_{\Z} W$.
\item The set $\{ L(\lambda) \ | \ \lambda \in \Irr_{\Z} W \}$ is a complete set of isomorphism classes of simple modules in $\grLmod{\rH_{\mbf{c}}(W)}$.
\item The set $\{ F(L(\lambda)) \ | \ \lambda \in \Irr_{\Z} W \}$ is a complete set of isomorphism classes of simple $\rH_{\mbf{c}}(W)$-modules.
\end{enumerate}
\end{prop}

Proposition \ref{prop:listofproperties} shows that there is a natural bijection, $\lambda \mapsto L(\lambda)$, between $\Irr W$ and $\Irr \rH_{\mbf{c}}(W)$. Therefore the blocks of $\rH_{\mbf{c}}(W)$ define a partition, which we call the \textit{block partition}, of the set $\Irr W$.

\section{Blocks for $G(m,d,n)$}\label{sec:blocksG(m,1,n)}

\subsection{}\label{wreathcherednik}
Let $m \geq 1, n >1$ be integers. Let $C_m$ be the cyclic group of order $m$. We fix a generator $g \in C_m$ and let $s_{ij}\in S_n$ denote the transposition which swaps $i$ and $j$. We denote by $s_j$ the simple transposition
swapping $j$ and $j+1$. The group $W=G(m,1,n)$ is the semidirect product $S_n \ltimes (C_m)^n$. We write $g_i^l$ for the
element $(1, \dots ,g^l, \dots 1) \in G(m,1,n)$ with $g^l$ in the $i$th place. Let $V = k^n$ be the reflection
representation of $G(m,1,n)$. Recall that we assume that $p \nmid |W|$, so in particular $p \neq 2$. Let $\eta \in k$ be a primitive $m$th root of unity. We fix a basis $\{ y_1, \dots ,y_n\}$ of $V$ so that \begin{align*}
g_i(y_j) = \begin{cases} \eta y_j\
\mathrm{if}\ i=j\\ y_j\ \mathrm{otherwise}
\end{cases} \mathrm{and}\ \sigma(y_j) = y_{\sigma(j)},
\end{align*}
for all $i, j$ and all $\sigma \in S_n$. Let $\{x_1, \dots ,x_n\} \in V^*$ be the dual basis. The conjugacy classes of reflections in $W$ are given by

\begin{align}\label{conj1} \{s_{ij} g_i^{-l} g_j^l: 0 \leq l \leq m-1\ \mathrm{and}\ i\neq j\}, \end{align}
and, for each $1 \leq l \leq m-1$,

\begin{align}\label{conj2} \{ g_j^l: 1 \leq j \leq n\}.\end{align}

The parameter ${\bc}$ is represented by $(\kappa, c_1, \dots ,c_{m-1}) \in k^n$, where ${\bc}(s_{ij} g_i^{-l} g_j^l) =
\kappa$ and ${\bc}(g_j^l)= c_l$. Using this notation, the definition of the rational Cherednik algebra becomes the following.

\begin{defn} Let $W=G(m,1,n)$. Then $\H_{\bc}(W)$ is the quotient of $T(V \oplus V^*) \rtimes W$ by the relations:
\begin{align*}
&[x_i , x_j] = 0,\ \ [y_i, y_j]=0, \\ &[y_i , x_i] = 1 - {\kappa} \sum_{l=0}^{m-1} \sum_{j \neq i}
s_{ij} g_i^{-l} g_j^l - \sum_{l=1}^{m-1} {c}_l(1-\eta^{-l}) g_i^l,\\ &[y_i, x_j] =
{\kappa}\sum_{l=0}^{m-1} \eta^{-l} s_{ij}
g_i^{-l} g_j^{l}.\end{align*}
\end{defn}

\subsection{The Dunkl-Opdam operators}\label{DOops} For all $1 \leq i \leq n$, we define elements in $\H_{\bc}(W)$:
\begin{align} z_i &= y_i x_i - \frac{1}{2} + {\kappa} \sum_{l=0}^{m-1} \sum_{1 \leq j < i} s_{ij}
g_i^l g_j^{-l} - \sum_{l=1}^{m-1} {c}_l \eta^{-l} g_i^l\label{DO1}\\ &= x_i y_i + \frac{1}{2} -
{\kappa} \sum_{l=0}^{m-1} \sum_{i < j} s_{ij} g_i^{l} g_j^{-l} - \sum_{l=1}^{m-1} {c}_l g_i^l
\label{DO2}.\end{align}

By \cite[Lemma 3.2]{BlocksGmdn}, $[z_i,z_j]=0$ for all $i,j$. Let $k[z_1, \dots,z_n]$ denote the algebra generated by the $z_is$. By (\ref{PBW}), this is
a polynomial algebra. We denote by $E_r$ and $P_r$ the $r$th elementary symmetric polynomial and $r$th power sum in the $z_i$, respectively. By convention, $E_0 = P_0 = 1$. We will use the following result from \cite[4.4]{BlocksGmdn}.

\begin{thm}
Keep notation as above. Then \[[E_r,x_1] = \sum_{1 < j_2 < \dots < j_r \leq n} x_1 z_{j_2} \dots z_{j_r}.\]
\end{thm}

\subsection{}\label{power_sums}
Let $E_r' = \frac{d}{dz_1}(E_r)$ and $P_r' = \frac{d}{dz_1}(P_r) = rz_1^{r-1}$. In this notation Theorem
\ref{DOops} reads \begin{align}\label{comm_symm} [E_r,x_1] = x_1 E_r'.\end{align} Recall Newton's formula: \begin{align}\label{Newton} rE_r = \sum_{i=1}^r
(-1)^{i-1}P_iE_{r-i}.\end{align} Applying $\frac{d}{dz_1}$ to (\ref{Newton}) we get \begin{align}\label{derive_E} rE_r' = \sum_{i=1}^r (-1)^{i-1}(P_i'
E_{r-i} + P_i E_{r-i}').\end{align} A straightforward induction argument using the fact that $E_r' = E_{r-1} -z_1E_{r-1}'$ shows that \begin{align}\label{derive_E2} E_r' = \sum_{i=0}^{r-1}
(-1)^{r-i-1} z_1^{r-i-1}E_{i}.
\end{align} We will need a simple preparatory lemma. Let $Q_n := (z_1+1)^n - z_1^n$.

\begin{lem}
We have \[ Q_{n+1} = \sum_{i=1}^{n} z_1^{n-i} Q_i + (n+1)z_1^n.\]
\end{lem}
\begin{proof}
The proof is by induction. The case $n=0$ is clear. For $n >0$, we have \[ Q_{n+1} = z_1 Q_n + Q_n + z_1^n,\] and the
induction step follows by a simple calculation which we leave to the reader.
\end{proof}

\begin{prop}
For $1\leq r \leq n$, \[ [P_r,x_1]= x_1 Q_r.\]
\end{prop}
\begin{proof}
The proof goes by induction on $r$. For $r=1$ we have $P_1 = E_1$ and $E_1' =1$. By Theorem \ref{DOops}, $[P_1, x_1]=x_1$.

Suppose $r>1$. Then \begin{align} (-1)^{r-1}[P_r,x_1] &= [rE_r + \sum_{i=1}^{r-1} (-1)^{i}P_iE_{r-i},x_1] \notag \\ \label{eqn1} &= r x_1 E_r' +
\sum_{i=1}^{r-1} (-1)^{i} (P_ix_1E_{r-i}' + x_1 Q_i E_{r-i}).\end{align} The first line follows from rewriting $P_r$ using (\ref{Newton}), and the second line follows from the induction hypothesis and (\ref{comm_symm}). Now (\ref{eqn1}) equals \begin{align}\notag r &x_1 E_r' + \sum_{i=1}^{r-1} (-1)^{i} (x_1P_iE_{r-i}'+ x_1
Q_i(E_{r-i}' + E_{r-i})) \\ \notag =&x_1 \left[ \sum_{i=1}^{r} (-1)^{i-1}(P_i'
E_{r-i} + P_i E_{r-i}') + \sum_{i=1}^{r-1} (-1)^{i} (P_iE_{r-i}'+  Q_i(E_{r-i}' + E_{r-i})) \right]\\ \notag
=& x_1 \left[ \sum_{i=1}^r (-1)^{i-1}P_i' E_{r-i} + \sum_{i=1}^{r-1} (-1)^{i} ( Q_i(E_{r-i}' + E_{r-i}) \right]\\ \label{eqn2} =& x_1
\left[ \sum_{i=1}^r (-1)^{i-1}iz_1^{i-1} E_{r-i} + \sum_{i=1}^{r-1} (-1)^{i} (Q_i(\sum_{k=0}^{r-i-1} (-1)^{r-i-k-1}
z_1^{r-i-k-1}E_{k}) + Q_iE_{r-i}) \right].
\end{align} Here the first line follows from the induction hypothesis, the second from (\ref{derive_E}) and the fourth from (\ref{derive_E2}). For a fixed $1 \leq l \leq r-1$, the coefficient of $E_{r-l}$ in (\ref{eqn2}) is equal to \[ (-1)^{l-1}lz_1^{l-1} +
\sum_{i=1}^{l-1} (-1)^{l-1} z_1^{l-i-1} Q_i  + (-1)^l Q_l, \] which equals zero by the lemma above. By the lemma above, the coefficient of $E_0$ is \[ (-1)^{r-1} r z^{r-1} + \sum_{i=1}^{r-1} (-1)^{r-1} z_1^{r-i-1} Q_i = (-1)^{r-1} Q_r. \] Therefore, \[ [P_r,x_1] = x_1 Q_r.\]
\end{proof}

\subsection{}\label{power_sums_central}
We can now provide some central elements in $\H_{\bf c}(W)$.

\begin{thm} For all $1 \leq r \leq n$,
\[ P_r(z_1^p-z_1, \dots ,z_n^p-z_n) = \sum_{i=1}^n (z_i^p - z_i)^r \in Z_{\bc}(W).\] Thus, $k [z_1^p-z_1, \dots ,z_n^p-z_n]^{S_n} \subset Z_{\bc}(W).$
\end{thm}
\begin{proof} We continue to write $P_m$ for the power sum in the variables $z_1, \dots ,z_n$. We have \[ \sum_{i=1}^n (z_i^p - z_i)^r = \sum_{i=1}^n \sum_{j=0}^r {r\choose j} (-1)^{r-j} z_i^{pj}z_i^{r-j} = \sum_{j=0}^r {r\choose j} (-1)^{r-j} P_{pj+r-j}.\] Therefore by Proposition \ref{power_sums}, \begin{align*} [\sum_{i=1}^n (z_i^p - z_i)^r,x_1] &= x_1 \sum_{j=0}^r {r\choose j} (-1)^{r-j} Q_{pj+r-j}\\ &= x_1 \sum_{j=0}^r {r\choose j} (-1)^{r-j} ( (z_1+1)^{pj+r-j} - z_1^{pj+r-j})\\ &= x_1 \sum_{j=0}^r {r\choose j} (-1)^{r-j} ( (z_1^p+1)^j (z_1+1)^{r-j} - (z_1^p)^jz_1^{r-j})\\ &= x_1\left[ (z_1^p+1) - (z_1+1) \right]^r - x_1 (z_1^p-z_1)^r = 0. \end{align*} The theorem now follows as in \cite[Theorem 3.4]{BlocksGmdn}. The element $P_r(z_1^p-z_1, \dots ,z_n^p-z_n)$ is symmetric in the $z_is$ so it commutes with any $\sigma \in S_n$, see \cite[Lemma 5.1]{Gr1}. Therefore this power sum commutes with $\sigma x_1 \sigma^{-1} = \sigma(x_1)$ for all $\sigma$, and in particular with each $x_i$. To see that the power sum also commutes with the $y_is$, we use the isomorphism $\psi: \H_{ \bf c}(W) \to \H_{{\bc'}}(W)$ from \cite[3.3]{BlocksGmdn}, where ${\bc'}$ is also defined. By definition $\psi(x_i)= y_{n-i+1}$ and $\psi(y_i) = x_{n-i+1}$, so that $\psi(z_i) = z_{n-i+1}$ for all $i$. Applying $\psi^{-1}$ to $[P_r(z_1^p-z_1, \dots ,z_n^p-z_n),x_i]$ shows that $[P_r(z_1^p-z_1, \dots ,z_n^p-z_n),y_i]=0$ for all $i$.
\end{proof}

\subsection{Blocks}\label{blocks} We use an identical argument to the proof of \cite[Theorem 5.5]{BlocksGmdn}, to determine the blocks of $\rH_c(W)$. We use freely the notation from \cite[$\S$5]{BlocksGmdn}. Let us first change our parameters. We define ${h}, {H}_0, \dots, H_{m-1} \in k$ via
\begin{align}
{h} = - {\kappa} \quad \mathrm{and} \quad -{c}_l(1-\eta^{-l}) = \sum_{j=0}^{m-1}
\eta^{-lj}{H}_j.\label{parameters}
\end{align}  
We denote by $\mathcal{P}(m,n)$ the set of $m$-multipartitions of $n$,
$\mathcal{P}(m,n) := \{ (\lambda^0, \dots ,\lambda^{m-1}): \sum_{i=0}^{m-1} |\lambda^i| = n\}$. The simple representations of $k W$ are labeled by the set $\mathcal{P}(m,n)$; over the complex numbers this is standard and the same holds for $k W$ by reduction, cf. \ref{RepsReduction}. In fact the construction from \cite{Push} is also valid over $k$, since $p$ does not divide $|W|$, so that we can find bases of irreducible $k W$-modules given by eigenvectors of Jucys-Murphy elements. Given a partition $\lambda$, the residue $\Res_{\lambda}(x^{-(\kappa^p - \kappa)})$ is an element in $\Z[k]$, the group ring of the additive group $(k,+)$. For $\blambda = (\lambda^0, \dots ,\lambda^{m-1}) \in \mc{P}(m,n)$ and $\ba = (a_1, \dots ,a_m) \in k^m$, let $\Res^{\ba} \blambda := \sum_{i=0}^{m-1} x^{({a}_i^p - a_i)} \Res_{\lambda^i}(x^{-(\kappa^p - \kappa)})$.

\begin{thm}
Let $\blambda, \bmu \in \mathcal{P}(m,n)$. Let $\ba = (0 , H_1, H_1 + H_2, \dots , H_1 + \dots + H_{m-1})$. Then $L_{\bc}(\blambda)$ and $L_{\bc}(\bmu)$ belong to the same block of $\rH_{\bc}(W)$ if and only if
$$
\Res^{\ba} \blambda = \Res^{\ba} \bmu.
$$
\end{thm}

\begin{proof}
Since this follows the proof of \cite[Theorem 5.5]{BlocksGmdn} closely, we shall sketch the argument, pointing out the necessary changes to \tit{loc. cit.} Using Theorem \ref{power_sums_central} together with Weyl's Theorem for invariants of symmetric groups (which is valid whenever $p \nmid |W|$, see \cite[Theorem 3.3.1]{Smith}), it is enough to calculate the characters of $P_r(z_1^p-z_1, \dots ,z_n^p-z_n)$ on each $L_{\bc}(\blambda)$. The characters are determined by choosing a simultaneous eigenvector $v_{\blambda} \in L_{\bc}(\blambda)$ for $z_1, \dots ,z_n$. The eigenvalues for $z_i^p - z_i$ are evaluated as in \cite[5.4]{BlocksGmdn}, and produce the desired combinatorial description.
\end{proof}

\subsection{Blocks for $G(m,d,n)$}\label{ss:G(m,d,n)blocks} The blocks of $\rH_c(W)$ for $W = G(m,d,n)$, where $d \mid m$ and either $n > 2$ or $n=2$ and $d$ is odd, can be calculated from Theorem \ref{blocks} by using Clifford theory as in \cite{CMpartitions}. The resulting description of the blocks is then analogous to that given in characteristic zero, see \cite[5.6]{BlocksGmdn}.

\section{Smoothness of centres for $G(m,1,n)$}\label{sec:smoothness}

\subsection{Dunkl embedding} 

Let $\alpha = \prod_{s\in S} \alpha_s$. Let $V_{\mathrm{reg}}$ denote the set $\{v\in V \mid \alpha(v) \neq 0\}$. Let $\mathcal{D}(V_{\mathrm{reg}})$ be the algebra of crystalline differential operators on $V_{\mathrm{reg}}$. For $y\in V$, let $D_y$ denote the Dunkl operator $\partial_y+\sum_{s\in\mathcal{S}}\frac{\mathbf{c}(s)}{1-\lambda_s}\frac{\alpha_s(y)}{\alpha_s}(s-1) \in \mathcal{D}(V_{\mathrm{reg}}) \rtimes W$. By \cite[page 280]{EG} there is an injective algebra morphism 
\[ 
\H_{\bc}(W) \hookrightarrow \mathcal{D}(V_{\mathrm{reg}}) \rtimes W; \quad w \mapsto w, x \mapsto x, y \mapsto D_y, 
\]
for all $w \in W$, $x\in V^*$ and $y\in V$. This embedding becomes an isomorphism after localizing $\H_{\bc}(W)$ at the Ore set $\{ \alpha^m \}_{m \ge 0}$. 

\subsection{Completions}

Let us recall the setup of \cite{BE}. Let $W'\subset W$ be the stabilizer of a point $b\in V$ and let $\overline{V}=V/V^{W'}$. For any point $b\in V$ we write $k[[V]]_b$ for the completion of $k[V]$ at $b$, and we write $\widehat{k[V]}_b$ for the completion of $k[V]$ at the $W$-orbit of $b$ in $V$. Note that we have $k[[V]]_0=\widehat{k[V]}_0$. For any finitely generated $k[V]$-module $M$, let
$$
\widehat{M}_b = \widehat{k[V]}_b\otimes_{k[V]} M.
$$ The completion $\widehat{\H}_{\bc}(W,V)_b$ is the algebra generated by $k\widehat{[V]}_b$, the Dunkl operators $D_y$ for $y\in V$, and the group $W$. Let ${\bc'}$ denote the restriction of $\bc$ to $S \cap W'$. The algebra $\widehat{\H}_{{\bc'}}(W',V)_0$ is then defined similarly. Let $P=\mathrm{Fun}_{W'}(W,\widehat{\H}_{\bc}(W',V)_0)$ be the set of $W'$-equivariant maps from $W$ to $\widehat{\H}_{\bc}(W',V)_0$. Let $Z(W,W',\widehat{\H}_{\bc}(W',V)_0)$ be the ring of endomorphisms of the right $\widehat{\H}_{\bc}(W',V)_0$-module $P$. The following proposition is proved over $\C$ in \cite[Theorem 3.2]{BE}, and has an identical proof over $k$.

\begin{prop}\label{BEiso}
There is an isomorphism of algebras
$$
\Theta: \widehat{\H}_{\bc}(W,V)_b\to Z(W,W', \widehat{\H}_{{\bc'}}(W',V)_0)
$$
defined as follows: for $f\in P$, $\alpha\in V^\ast$, $a\in V$, $u\in W$,
\begin{eqnarray*}
(\Theta(u)f)(w)&=&f(w u),\\
(\Theta(x_{\alpha})f)(w)&=&(x^{(b)}_{w\alpha}+\alpha(w^{-1}b))f(w),\\
(\Theta(y_a)f)(w)&=&y^{(b)}_{wa}f(w)+\sum_{s\in\mathcal{S}, s\notin
W'}\frac{c_s}{1-\lambda_s}\frac{\alpha_s(wa)}{x^{(b)}_{\alpha_s}+\alpha_s
(b)}(f(sw)-f(w)),
\end{eqnarray*}
where $x_\alpha \in V^\ast\subset
\H_{\bc}(W,V)$, $x^{(b)}_{\alpha}\in V^\ast\subset \H_{{\bc'}}(W',V)$,
$y_a\in V \subset \H_{\bc}(W,V)$ and $y_a^{(b)}\in V \subset \H_{{\bc'}}(W',V)$.
\end{prop}

Let $e_{11} \in Z(W,W', \widehat{\H}_{{\bc'}}(W',V)_0)$ be defined by $(e_{11} \cdot f)(w) = f(w)$ if $w \in W'$ and $(e_{11} \cdot f)(w) = 0$ otherwise. Then $e_{11}$ is a primitive idempotent and $\widehat{\H}_{{\bc'}}(W',V)_0 \cong e_{11} Z(W,W', \widehat{\H}_{{\bc'}}(W',V)_0) e_{11}$. The bimodule $e_{11} Z(W,W', \widehat{\H}_{{\bc'}}(W',V)_0)$ yields a Morita equivalence between $\widehat{\H}_{{\bc'}}(W',V)_0$ and $Z(W,W', \widehat{\H}_{{\bc'}}(W',V)_0)$.

\subsection{}\label{ss:small_modules}

Suppose that $M$ is a finite dimensional $\H_{\bc}(W,V)$-module and let $\supp M$ denote the support of $M$ as a $k[V]^W$-module. If $\supp M$ is the $W$-orbit of some point $b \in V$, then $M$ is naturally a $\widehat{\H}_{c}(W,V)_b$-module. In particular, the support of any simple $\H_{\bc}(W,V)$-module is a single $W$-orbit and so is a $\widehat{\H}_{\bc}(W,V)_b$-module. Let $M$ be a $\widehat{\H}_{\bc}(W,V)_b$-module, then denote by $\Theta^* M$ the $Z(W,W', \widehat{\H}_{{\bc'}}(W',V)_0)$-module given by transporting the $\widehat{\H}_{\bc}(W,V)_b$-action on $M$ via $\Theta$.

\begin{defn}
An $\H_{\bc}(W, V)$-module $M$ is called {\it small} if $\dim M < |W|p^n$.
\end{defn}

\begin{prop}\label{prop:completesmall}
A $\H_{\bc}(W,V)$-module $M$ is small if and only if $e_{11} \Theta^* (M)$ is a small $\H_{{\bc'}}(W',V)$-module.
\end{prop}

\begin{proof}
Note first that $e_{11} \Theta^* (M)$ is a $\H_{{\bc'}}(W',V)$-module by restricting the $\widehat{\H}_{\bc'}(W',V)_0$-action. Let $g_1 = 1 , \dots , g_t$ be right coset representatives of $W'$ in $W$. Then $\Theta^* (M) = \bigoplus_{i=1}^t g_i e_{11} g_i^{-1} \Theta^* M$, with $\dim_k g_i e_{11} g_i^{-1} \Theta^* M = \dim_k g_j e_{11} g_j^{-1} \Theta^* M$ for all $i,j$. In particular, $\dim M < |W|p^n$ if and only if $\dim e_{11} \Theta^* (M) < |W|p^n/t = |W'|p^n$.
\end{proof}

Let $Z_{\bc}(W,V)$ denote the centre of $\H_{\bc}(W,V)$.

\begin{cor}
Keep notation as above. Then $Z_{\bc}(W,V)$ is smooth if and only if $\H_{{\bc'}}(W',V)$ has no small modules $M$ with $\supp M = 0$ for all parabolic subgroups $W' \subseteq W$.
\end{cor}

\begin{proof}
By Proposition \ref{Cherednik_props} (3), $Z_{\bc}(W,V)$ is smooth if and only if $\H_{\bc}(W,V)$ has no small modules. Let $M$ be a small, simple $\H_{\bc}(W,V)$-module and choose $b \in V$ such that the support of $M$ equals the $W$-orbit of $b$. Then $M$ extends to a $\widehat{\H}_{\bc}(W,V)_b$-module and Proposition \ref{prop:completesmall} says that $e_{11} \Theta^* (M)$ is a small $\H_{{\bc'}}(W',V)$-module, where $W'$ is the stabilizer of $b$ in $W$. Moreover, the support of $e_{11} \Theta^* (M)$ is $0$. Conversely, if there exists a small $\H_{{\bc'}}(W',V)$-module $M'$ supported on $0$ then this extends to a $\widehat{\H}_{{\bc'}}(W',V)_0$-module. Since $\widehat{\H}_{{\bc'}}(W',V)_0$ is Morita equivalent to $Z(W,W', \widehat{\H}_{{\bc'}}(W',V)_0) \cong \widehat{\H}_{\bc}(W,V)_b$, there exists some $\widehat{\H}_{{\bc}}(W,V)_b$-module $M$ such that $M' \simeq e_{11} \Theta^* (M)$. Proposition \ref{prop:completesmall} says that $M$ is actually a small $\H_{\bc}(W,V)$-module.
\end{proof}

\subsection{} Let $M$ be a simple $\H_{\bc}(W,V)$-module with $\supp M=0$. We can define an induced module (cf. \ref{rrca}) as follows. Let $B = k[V] \rtimes W$ and let $E \in \Irr W$. Extend the $W$-action on $E$ to a $B$-action by letting $f\in k[V]$ act by $f(0)$. Define \[
\Delta(E) = \H_{\bc}(W,V) \o_{B} E.
\]
By the support condition on $M$, the subspace $M_0 = \{ m \in M \ | \ V^* \cdot m = 0 \}$ of $M$ is non-zero. It is a $W$-submodule of $M$ and we may assume without loss of generality that $E \subseteq M_0$. Therefore there is a surjective homomorphism $\Delta(E) \twoheadrightarrow M$, which maps $1 \otimes E \subset \Delta(E)$ to  $E \subseteq M_0$ in the obvious way. Recall that $\H_{\bc}(W,V)$ has a $\Z$-grading. We make $E$ into a graded $B$-module by setting $E_0 = E$ and $E_i = 0$ for all $i \neq 0$. We give $\Delta(E)$ a $\Z$-graded module structure by inducing the graded structure on $E$. Since $M$ is simple, there is some $a \in (V^*)^{(1)}/W$ such that $\mf{m}_a \cdot M = 0$, where $\mf{m}_a$ denotes the maximal ideal corresponding to $a$. Define 
$$
\Delta(a,E) = (\H_{\bc}(W,V) \o_{B} E) / \mf{m}_a \cdot (\H_{\bc}(W,V) \o_{B} E).
$$ 
By our choice of $a$, there is also a surjective homomorphism $\Delta (a,E) \twoheadrightarrow M$. The grading on $\Delta(E)$ induces filtrations on $\Delta(a,E)$ and $M$. We use the notation $\gr_{\Z}$ to denote the associated graded objects with respect to these filtrations.

Recall that a morphism $f : M \rightarrow N$ between filtered $\H_{\bc}(W,V)$-modules is called a \tit{strictly filtered} morphism if $f(F_i M) = F_i N \cap f(M)$ for all $i \in \Z$. The functor $\gr_{\Z}$ is exact on short exact sequences of strictly filtered morphisms. Note also that the surjection $\Delta(a,E) \twoheadrightarrow M$ is strictly filtered by definition. The filtration on $\Delta(E)$ is both exhaustive and separating, therefore the same is true of $\Delta(a,E)$. However, this module is finite dimensional therefore we have $F_i \Delta(a,E) = 0$ and $F_{-i} \Delta(a,E) = \Delta(a,E)$ for $i \gg 0$. 

\begin{prop}\label{prop:assocgraded}
The $\H_{\bc}(W,V)$-module $\gr_{\Z} M$ has dimension $\dim_{k} M$ and is annihilated by $(k[V^{(1)}]^W \o k[(V^*)^{(1)}]^W)_+$.
\end{prop}

\begin{proof}
Denote by $F_i$ the $i$th piece of the filtration on $M$. By construction, $F_{\bullet}$ is a decreasing filtration with $F_1 = 0$ and $F_{-i} = M$ for all $i \gg 0$. Thus $\dim_k \gr_{\Z} M = \sum_{i=0}^\infty \dim_k F_{-i}/F_{-i+1} = \dim_k M$.

For the second statement, we show that $\gr_{\Z} \Delta(a,E) \cong \Delta(0,E)$. The claim then follows since there is a surjective map $\Delta (0,E) \to \gr_{\Z} M$. Consider the short exact sequence:
\[ 
0 \to \mf{m}_a \cdot (\H_{\bc}(W,V) \o_{B} E) \to \Delta(E) \to \Delta (a,E) \to 0,
\] 
where the left-hand term is given the induced filtration. By the PBW theorem and the Nullstellensatz, $\gr_{\Z} \mf{m}_a \cdot (\H_{\bc}(W,V) \o_{B} E) = k[(V^*)^{(1)}]^W_+ \cdot (\H_{\bc}(W,V) \o_{B} E)$. Since $\Delta(E)$ is graded, $\gr_{\Z} \Delta (E) = \Delta(E)$. Thus $\gr_{\Z} \Delta(a,E) \cong \Delta(0,E)$.
\end{proof}

\begin{cor}\label{cor:equivsmooth}
The following are equivalent:
\begin{enumerate}
\item $Z_{\bc}(W,V)$ is smooth;
\item $\H_{{\bc'}}(W',\overline{V})$ has no small modules $M$ with $\supp M = 0$ for all parabolic subgroups $W' \subseteq W$;
\item $\rH_{{\bc'}}(W',\overline{V})$ has no small modules $M$ for all parabolic subgroups $W' \subseteq W$;
\item The blocks of $\rH_{{\bc'}}(W',\overline{V})$ are singletons for all parabolic subgroups $W' \subseteq W$.
\end{enumerate}
\end{cor}
\begin{proof}
There are isomorphisms of algebras $\H_{{\bc'}}(W', {V}) \simeq \H_{{\bc'}}(W',\overline{V}) \o \dd(V^{W'})$ and $\rH_{{\bc'}}(W', {V}) \simeq \rH_{{\bc'}}(W',\overline{V}) \o \overline{\dd}(V^{W'})$, where $\overline{\dd}(V^{W'}) := \dd(V^{W'}) / \langle (k[(V^{W'})^{(1)}] \o k[((V^{W'})^*)^{(1)}])_+ \rangle$. In particular, (2) and (3) are equivalent to
\begin{itemize}
\item[(2')] $\H_{{\bc'}}(W',{V})$ has no small modules $M$ with $\supp M = 0$ for all parabolic subgroups $W' \subseteq W$;
\item[(3')] $\rH_{{\bc'}}(W',{V})$ has no small modules $M$ for all parabolic subgroups $W' \subseteq W$,
\end{itemize} 
respectively, where we have used the fact that $\dd(V^{W'})$ is an Azumaya algebra. The equivalence of (1) and (2') is Corollary \ref{ss:small_modules}. Clearly (2') implies (3'). For the converse, suppose that $\H_{{\bc'}}(W',V)$ has a small module $M$ with $\supp M = 0$ for some parabolic subgroup $W' \subseteq W$. Then by the proposition, $\gr_{\Z} M$ is a small module for $\rH_{{\bc'}}(W',V)$.

The equivalence of (3) and (4) follows from an identical argument to \cite{Baby}.
\end{proof}

\begin{rem}\label{rem:remark1}
The proofs above are valid, mutando mutandae, for $t=0$ and char $k = 0$.
\end{rem}

\subsection{Smoothness of centres of $\H_{\bc}(G(m,1,n))$}\label{ss:smoothcentres}

Let $m,n$ be positive integers and assume that $n > 1$. Let $W = G(m,1,n)$. In this section we give a proof of Corollary \ref{cor:GMoneNsmooth}. The idea is to use Corollary \ref{cor:equivsmooth} (4) together with the results of section \ref{sec:blocksG(m,1,n)} to determine for which $\bc$ the centre $Z_{\bc}(W)$ is smooth. Recall that the parabolic subgroups of $W$ are of the form \[ S_{k_1} \times \dots \times S_{k_t} \times G(m,1,n'),\quad k_1 + \dots + k_t + n' \leq n. \] Here $S_{k_i}$ denotes the symmetric group on $k_i$ letters, and $S_0 = G(m,1,0) = \{ \id \}$ by definition. For such a parabolic subgroup, the representation $\overline{V}$ is the reflection representation. Recall that $\bc = (\kappa, c_1, \dots , c_{m-1})$. Recall the parameters $a_1, \dots ,a_m$ from Theorem \ref{blocks}. Let $\mc{C}_{m,n}$ denote the set of all $\bc$ such that either 
\[ 
\kappa \in \mathbb{F}_p, \quad \textrm{ or } \quad a_i - a_j \pm C \kappa \in \mathbb{F}_p,
\]
for some $\ 0 \le i \neq j \le m-1$ and integer $C$ such that $0 \leq C \leq n-1$. Thus $\mc{C}_{m,n}$ is a finite union of hyperplanes in $k^m$.

\begin{thm}
The algebra $Z_{\bc}(W)$ is smooth if and only if $\bc \notin \mc{C}_{m,n}$.
\end{thm}
\begin{proof}
Let us first suppose that $\bc \in \mc{C}_{m,n}$. Recall that, for $\blambda \in \mc{P}(m,n)$, $\Res^{\ba} \blambda := \sum_{i=0}^{m-1} x^{({a}_i^p - a_i)} \Res_{\lambda^i}(x^{-(\kappa^p - \kappa)})$. Suppose that $\kappa \in \mathbb{F}_p$. Then $\kappa^p - \kappa = 0$ and so $\Res^{\ba}((n),\emptyset, \dots ,\emptyset) = \Res^{\ba}((n)^t,\emptyset, \dots ,\emptyset)$. Thus by Theorem \ref{blocks}, $\rH_{{\bc}}(W,{V})$ has a non-singleton block, and so $Z_{\bc}(W)$ is singular, Corollary \ref{cor:equivsmooth}. If $m>1$ and $a_i - a_j - C \kappa \in \mathbb{F}_p$ for some $0 \leq C \leq n-1$, then  
\[
a_i^p - a_i = a_j^p - a_j + C(\kappa^p - \kappa).
\]
Without loss of generality $i=1$ and $j=2$, and then 
$$
\Res^{\ba}((n),\emptyset,\emptyset, \dots ,\emptyset) = \Res^{\ba}(\emptyset,(n-C,1^C),\emptyset, \dots ,\emptyset) = x^{a_1^p - a_1} \sum_{i = 0}^{n-1} x^{-i (\kappa^p - \kappa)}.
$$
Thus $\rH_{{\bc}}(W,{V})$ has a non-singleton block and $Z_{\bc}(W)$ is singular. A similar argument applies in the case $-n+1 \leq C \leq 0$.

Suppose now that $\bc \notin \mc{C}_{m,n}$. We first show that $\rH_{{\bc}}(W,{V})$ has only singleton blocks. For a contradiction, suppose that there exist distinct $\blambda, \bmu \in \mc{P}(m,n)$ such that $\Res^{\ba} \blambda = \Res^{\ba} \bmu$. Since $\kappa \notin \mathbb{F}_p$, $x^{-(\kappa^p-\kappa)} \neq 1$. Each box $b$ in the Young diagram of $\lambda^i$ contributes $x^{({a}_i^p - a_i)} x^{-\cont (b) (\kappa^p - \kappa)}$ to $\Res^{\ba} \blambda$. Since $\blambda \neq \bmu$, there exist $1 \leq i < j \leq N$ and boxes $b \in \lambda^i$, $b' \in \mu^j$ such that $x^{({a}_i^p - a_i)} x^{-\cont (b) (\kappa^p - \kappa)} = x^{({a}_j^p - a_j)} x^{-\cont (b') (\kappa^p - \kappa)}$. Now $\cont (b)$ and $\cont (b')$ are integers such that $-|\lambda^i|+1 \leq \cont (b) \leq |\lambda^i| - 1$ and $-|\mu^j|+1 \leq \cont (b') \leq |\mu^j| - 1$. Thus $a_i^p - a_i = a_j^p - a_j + C (\kappa^p - \kappa)$ for some $-|\lambda^i| - |\mu^i| + 1 \leq C \leq |\lambda^i| + |\mu^i| - 1$. This means that $-a_i + a_j - C \kappa \in \mathbb{F}_p$, and so $\bc \in \mc{C}_{m,n}$, a contradiction.

We now want to prove the stronger statement that $\bc \notin \mc{C}_{m,n}$ implies that $Z_{\bc}(W)$ is smooth. Note that $\mc{C}_{m',n'} \subset \mc{C}_{m,n}$ for all $m' = 1$ or $m$, $n'< n$. Therefore $\bc \notin \mc{C}_{m,n}$ implies that $\bc' \notin \mc{C}_{m',n'}$ for all $m' = 1$ or $m$, $n'< n$. By the description of parabolic subgroups $W' \subset W$ given above, the previous paragraph implies $\rH_{{\bc}}(W',\overline{V})$ has only singleton blocks for all $W'$. By Corollary \ref{cor:equivsmooth} this implies that $Z_{\bc}(W)$ is smooth.
\end{proof}

\begin{remark}\label{rem:symsing}
\begin{enumerate}

\item In the case $m=1$, $W$ is just the symmetric group $S_n$. Although $V=k^n$ is not the reflection representation, we have $\H_{\kappa}(S_n, k^n) \cong \H_{\kappa}(S_n, k^{n-1}) \o \dd(\mathbb{A}^1)$, where $k^{n-1}$ now denotes the reflection representation of $S_n$. The set $\mc{C}_{1,n}$ is then identified with $\mathbb{F}_p \subset k$.

\item Note that the proof of the above theorem shows that for $W=G(m,1,n)$, the centre $Z_{\bc}(W)$ is smooth if and only if $\rH_{{\bc}}(W,{V})$ has only singleton blocks.
\end{enumerate}
\end{remark}

\section{Degenerations}

In this section we describe, based on Cartier's Theorem, the category of $\dd(V) \rtimes W$-module with $p$-curvature zero. This will allow us to prove Theorem \ref{thm:classify}.

\subsection{$p$-curvature}

Fix a basis $x_1, \ds, x_n$ of $V^*$ and $\pa_1,\ds, \pa_n$ of $V$ such that $\pa_i(x_j) = \delta_{i,j}$. Let
$$
A = k[\pa_1, \ds, \pa_n,x_1^p, \ds, x_n^p]\quad \textrm{and} \quad \Spec A = T^{*,(1)} V.
$$
The centre of $\dd(V)$ embeds in $A$ and we write $\pi : T^{*,(1)} V \ra (T^{*} V)^{(1)}$ for the corresponding finite morphism. The group $W$ acts on $T^{*,(1)} V$ and $(T^* V)^{(1)}$ such that the map $\pi$ is $W$-equivariant and satisfies 
$$
W_\zeta := \Stab_W(\zeta) = \Stab_W(\pi(\zeta)), \quad \forall \ \zeta \in T^{*,(1)} V.
$$
For fixed $\zeta \in T^{*,(1)} V$ and $\lambda \in \Irr_k(W_\zeta)$, we define $V_\zeta(\lambda) := \Ind_{A \rtimes W_\zeta}^{\dd(V) \rtimes W} \lambda$, where $A$ acts on $\lambda$ via the character $\zeta$.

\begin{lem}\label{prop:classifysimples}
Fix $\zeta \in T^{*,(1)} V$ and $\lambda \in \Irr_k(W_\zeta)$. Then
\begin{enumerate}
\item the $\dd(V) \rtimes W$-module $V_\zeta (\lambda)$ is simple;
\item $V_{\zeta_1} (\lambda_1) \simeq V_{\zeta_2} (\lambda_2)$ if and only if $\zeta_2 \in W \cdot \zeta_1$ and, moreover, if $w \zeta_1 = \zeta_2$ then $\lambda_1 \simeq \lambda_2$ via the conjugation isomorphism $w : W_{\zeta_1} \rightsim W_{\zeta_2}$;
\item Every simple $\dd(V) \rtimes W$-module is isomorphic to $V_\zeta (\lambda)$ for some $\zeta$ and $\lambda$.
\end{enumerate}
\end{lem}

\begin{proof}
Considered as an $A$-module, $V_\zeta (\lambda) = \bigoplus_{\eta \in W \cdot \zeta} (V_\zeta (\lambda))_{\eta}$ where $(V_\zeta (\lambda))_{\eta}$ is set-theoretically supported at $\eta$. Each $(V_\zeta (\lambda))_{w \zeta}$ is a $\dd(V) \rtimes (w W_\zeta w^{-1})$-submodule of $V_\zeta (\lambda)$ and $V_\zeta (\lambda)$ will be a simple $\dd(V) \rtimes W$-module if and only if $(V_\zeta (\lambda))_{\zeta} = \Ind_{A \rtimes W_\zeta}^{\dd(V) \rtimes W_\zeta} \lambda$ is a simple $\dd(V) \rtimes W_\zeta$-module. If $\zeta = (a, \alpha)$ with $a \in V^{(1)}$ then we write $(b,\alpha)$ with $b \in V$ for the unique closed point in $\pi^{-1}(\zeta)$.  Applying the $W_\zeta$-equivariant automorphism $x_i \mapsto x_i - \langle x_i, b\rangle$ and $\pa_j \mapsto \pa_j - \langle \pa_j,\alpha \rangle$ to $\dd(V)$, we may assume without loss of generality that $\zeta = 0$ and $W_\zeta = W$. Let $\delta^0 = \Ind_{A}^{\dd(V)} k$ be the unique simple $\dd(V)$-module supported at $0 \in (T^* V)^{(1)}$; simplicity of $\delta^0$ follows from the fact that $\dd(V)$ is Azumaya of rank $p^{2n}$ and the dimension of $\delta^0$ is $p^n$. Then $V_0(\lambda) = \delta^0 \o \lambda$, with $W$ acting diagonally. The module $V_0(\lambda)$ is simple: let $v_1, \ds, v_{p^n}$ be a basis of $\delta^0$ such that $v_1 = 1$ and choose any $0 \neq l = \sum_i v_i \o l_i \in V_0(\lambda)$. Choose an $i$ such that $l_i \neq 0$. Since $\dd(V)$ surjects onto $\End_k(\delta^0)$, there is some $D \in \dd(V)$ such that $D \cdot v_j = 0$ for all $j \neq i$ and $D \cdot v_i = v_1$. Then $D \cdot l = 1 \o l_i$ and we have $W \cdot (D \cdot l) = 1 \o \lambda$. Hence $\dd(V) \rtimes W \cdot l = V_0(\lambda)$. To show that the various $V_0(\mu)$ are non-isomorphic, note that
$$
\dim_k \Hom_{\dd(V) \rtimes W}(V_0(\lambda),V_0(\mu)) = \dim_k \Hom_{A \rtimes
W}(\lambda, V_0(\mu)) = \delta_{\lambda,\mu}
$$
because the space $\{ v \in \delta^0 \ | \ \pa_i \cdot v = 0 \ \forall \ i \}$ is one-dimensional. Arguing geometrically as above shows the second claim of the statement. The final statement is clear just by considering the socle, as a $A$-module, of an simple $\dd(V) \rtimes W$-module.
\end{proof}

\subsection{} Now we require a special case, Proposition \ref{prop:Cartier}, of a classical result by Cartier on $\dd$-module with zero $p$-curvature. We follow the presentation given in \cite[\S 5]{KatzNilpotent}. Recall that if $D \in \Der (V) \subset \dd(V)$ is a derivation, then $D^p$ also acts as a derivation and we write $D^{[p]}$ for this derivation so that $D^p - D^{[p]}$ acts trivially on $k [V]$.

\begin{defn}
Let $M$ be a finitely generated $\dd(V) \rtimes W$-module. The \tit{$p$-curvature} of $M$ is the map $\psi : \Der(V) \ra \End_k(M)$ given by $\psi(D) = \rho(D)^p - \rho(D^{[p]})$, where $\rho : \dd(V) \rtimes W \ra
\End_k(M)$ is the action map. We say that $M$ has \tit{zero $p$-curvature} if $\psi = 0$.
\end{defn}

Denote by $\mc{D}_0$ the full subcategory of $\Lmod{\dd(V) \rtimes W}$ consisting of modules with zero $p$-curvature.

\begin{prop}\label{prop:Cartier}
Let $M \in \Lmod{\dd(V) \rtimes W}$. Let $V^{(1)} \subset (T^* V)^{(1)}$ be the zero section (defined by $\pa_1^p = \ds = \pa^p_n = 0$).
\begin{enumerate}
\item The module $M$ has zero $p$-curvature if and only if $M$ is scheme-theoretically supported on $V^{(1)}$ when considered as a $Z(\dd(V))$-module i.e. $\pa_i^p \cdot M = 0$ for all $i$.
\item The ``horizontal sections" functor $DR : \mc{D}_0 \ra \Lmod{k[V^{(1)}] \rtimes W}$,
$$
DR(M) := M^{\nabla} = \{ m \in M \ | \ \pa_i \cdot m = 0 \ \forall \ i \},
$$
is an equivalence of categories with quasi-inverse $DR^{\perp} : \Lmod{k[V^{(1)}] \rtimes W} \ra \mc{D}_0$,
$$
DR^{\perp}(N) = \Ind_{A \rtimes W}^{\dd(V) \rtimes W} N,
$$
where $A$ acts on $N$ via the morphism $A \ra k[V^{(1)}]$, $\pa_i \mapsto 0$.
\item The equivalence $DR$ restricts to an equivalence of graded categories $\gr \mc{D}_0 \ra \grLmod{k[V^{(1)}] \rtimes W}$.
\end{enumerate}
\end{prop}

The following operators where introduced in \cite[(5.1.2)]{KatzNilpotent}. Their properties can be verified by direct calculation.

\begin{lem}\label{lem:Poperator}
Let $\dd(x)$ be the first Weyl algebra and $M$ a $\dd(x)$-module with zero $p$-curvature.
\begin{enumerate}
\item Define $P = \sum_{i = 0}^{p-1} \frac{(-x)^i}{i !} \pa^i \in \dd(x)$. Then $P$ defines a $k[x^p]$-linear operator on $M$ such that $P(M) \subset M^{\nabla}$, $P |_{M^{\nabla}} = \id$ and $P^2 = P$.
\item Define the map $T : M \ra M$ by
$$
 m \mapsto \sum_{i = 0}^{p-1} \frac{x^i}{i !} P(\pa^i \cdot m).
$$
Then $T = \id_M$.
\end{enumerate}
\end{lem}

\begin{proof}[Proof of Proposition \ref{prop:Cartier}]

\begin{enumerate}
\item It is shown in \cite[Proposition 5.2]{KatzNilpotent} that $\psi$ is $p$-linear i.e. $\psi(f_1 D_1 + f_2 D_2) = f^p_1 \psi(D_1)+ f^p_2 \psi(D_2)$ for all $f_i \in k[V]$ and $D_i \in \Der(V)$. Since $\psi(\pa_i) = \rho(\pa_i)^p$ for all $i$ it is clearly necessary that $\rho(\pa_i)^p = 0$, i.e. $\pa_i^p \cdot M = 0$, in order for $M$ to have zero $p$-curvature. On the other hand, every element $D \in \Der(V)$ can be expressed as $D = \sum_{i = 1}^n f_i \pa_i$ for some $f_i \in k[V]$; thus
$$
\psi(D) = \sum_{i = 1}^n f_i^p \pa_i^p
$$
and $\psi(D) \cdot M = 0$ if $\pa_i^p \cdot M = 0$ for all $i$.

\item  As in Lemma \ref{lem:Poperator}, define $P_i = \sum_{j = 0}^{p-1} \frac{(-x_i)^j}{j !} \pa_i^j$ and set $P = \prod_{i = 1}^n P_i$. Let $M \in \mc{D}_0$. Then Lemma \ref{lem:Poperator} implies $P$ defines a $k[V^{(1)}]$-linear operator on $M$ such that $DR(M) = P(M)$, $P_{M^{\nabla}} = \mathsf{id}$ and $P^2 = P$. Note that the subspace $M^{\nabla}$ of $M$ is a $W$-submodule of $M$. Therefore we define
$$
\tilde{P} = \frac{1}{|W|} \sum_{w \in W} w \cdot P,
$$
so that $\tilde{P}$ is a $W$-equivariant projection onto $M^{\nabla}$. Let
$$
0 \ra M_1 \ra M_2 \stackrel{\phi}{\ra} M_3 \ra 0
$$
be a short exact sequence in $\mc{D}_0$. Since $\tilde{P} \in \dd(V)$, we have $\tilde{P} \circ \phi = \phi \circ \tilde{P}$. Therefore $DR(\phi) : DR(M_2) \ra DR(M_3)$ is surjective and the left exact functor $DR$ is actually exact. Similarly, since $\dd(V)  \rtimes W$ is flat over $A \rtimes W$, $DR^{\perp}$ is also an exact functor. If $(\pa_1, \ds, \pa_n)$ is the left ideal of $\dd(V) \rtimes W$ generated by $\pa_1, \ds, \pa_n$, then $DR^{\perp}(k[V^{(1)}] \rtimes W) = \dd(V) \rtimes W / (\pa_1, \ds, \pa_n)$ and
$$
DR \left( \frac{\dd(V) \rtimes W}{(\pa_1, \ds, \pa_n)} \right) = \frac{A \rtimes W + (\pa_1, \ds, \pa_n)}{(\pa_1, \ds, \pa_n)} \simeq k[V^{(1)}] \rtimes W
$$
as a $k[V^{(1)}] \rtimes W$-module. Since $DR \circ DR^{\perp}(k[V^{(1)}] \rtimes W) = k[V^{(1)}] \rtimes W$ and the functor $DR \circ DR^{\perp}$ is exact, for each $N \in \Lmod{k[V^{(1)}] \rtimes W}$ we get the standard diagram
$$
\xymatrix{
k[V^{(1)}] \rtimes W^k \ar[r] \ar[d] & k[V^{(1)}] \rtimes W^l \ar[r]\ar[d]  & N \ar[r] \ar[d] & 0 \\
DR \circ DR^{\perp} (k[V^{(1)}] \rtimes W)^k \ar[r] & DR \circ DR^{\perp}(k[V^{(1)}] \rtimes W)^l \ar[r] & DR \circ DR^{\perp}(N) \ar[r] & 0
 }
$$
where the first two vertical morphisms are isomorphisms. This implies that the third vertical morphism is also an isomorphism. Hence the natural transformation $1 \ra DR \circ DR^{\perp}$, coming from the fact that $DR^{\perp}$ is left adjoint to $DR$, is an isomorphism. Now take $M \in \mc{D}_0$ and consider the natural morphism $DR^{\perp} \circ DR (M) \ra M$. Since $DR$ is conservative, the fact that  $DR \circ DR^{\perp} = 1$ implies that this morphism is injective. On the other hand, if we define $T = \prod_{i = 1}^{n} T_i$, where $T_i : M \ra M$,
$$
T_i(m) = \sum_{j = 0}^{p-1} \frac{x^j_i}{j !} P_i(\pa^j_i \cdot m),
$$
then Lemma \ref{lem:Poperator} implies that $T = \mathsf{id}_M$. This proves that $DR^{\perp} \circ DR (M) \ra M$ is surjective.

\item It is straight-forward to see that $DR$ and $DR^{\perp}$ send graded modules to graded modules.

\end{enumerate}

\end{proof}


\subsection{} Let $\mathbf{1}$ denote the trivial $W$-module. The following observation will be required later.

\begin{lem}\label{lem:tensorreg}
Let $N$ be a $k[V^{(1)}] \rtimes W$-module such that $DR^{\perp}(N)$ is isomorphic to $p^n$ copies of the regular representation as a $W$-module. Then $N$ is isomorphic to the regular representation as a $W$-module.
\end{lem}

\begin{proof}
As a $W$-module, $DR^{\perp}(N) \simeq V_0(\mathbf{1}) \o N$. Therefore it suffices to show that the Brauer character $\chi$ of $V_0(\mathbf{1})$ satisfies $\chi(w) \neq 0$ for all $w$ in $W$. Recall that $V_0(\mathbf{1}) \simeq k[V^*] / \langle k[(V^*)^{(1)}]_+ \rangle$. Let $\lambda_1, \ds, \lambda_n$ be the eigenvalues of $w$ on $V$. Since we are calculating the Brauer character of $V_0(\mathbf{1})$, we assume that $\lambda_i \in \C$ for all $i$. Then
$$
\chi(g) = \mathrm{Tr}(g,V_0(\mathbf{1})) = \prod_{i = 1}^n \frac{1 - \lambda^p_i}{1 - \lambda_i}.
$$
Since $\chi(g) \in \C$ and $p$ does not divide $|W|$, the product on the right hand side is non-zero.
\end{proof}

\subsection{Lattices}

In this section let $k$ be an arbitrary algebraically closed field. We let $H$ denoted a $\Z$-graded $k[x]$-algebra such that $H$ is a finite, free $k[x]$-module, where $k[x]$ is graded with $\deg (x) = 0$. Write $\ell = \Spec k[x]$ and $K = k(x)$. For $\alpha \in k$, we denote by $H_{\alpha}$ the specialization $H \o_{k[x]} k_{\alpha}$ of $H$ at $\alpha$. We assume that there exists a finite set $I$ and $H$-modules $\{ \Delta(\lambda) \ | \lambda \in I \}$ such that
\begin{enumerate}
\item The module $\Delta(\lambda)$ is graded and free as a $k[x]$-module.
\item For all $\alpha \in \ell$, there is a bijection $\Irr H_\alpha \simeq I$ such that the simple, graded $H_{\alpha}$-module $L_{\alpha}(\lambda)$ is the unique simple quotient of $\Delta_{\alpha}(\lambda)$, the specialization of $\Delta(\lambda)$ at $\alpha$. 
\item There is a bijection $I \simeq \Irr H_K$ such that the simple, graded $H_K$-module $L_K(\lambda)$ is the unique simple quotient of $\Delta_K(\lambda)$.
\end{enumerate}

Let $M$ be a $H$-module. We say that $M$ is a \textit{$H$-lattice} if it is free, of finite rank, as a $k[x]$-module. Note that if $M$ is a $H$-lattice and $N$ a $H$-submodule, then $N$ is a $H$-lattice because $k[x]$ is a principal ideal domain. 

\begin{lem}\label{lem:compseries}
There exists a composition series $0 = \Delta^0(\lambda) \subset \cdots \subset \Delta^r(\lambda) = \Delta(\lambda)$ of graded $H$-lattices such that if $L^i(\lambda) = \Delta^i(\lambda) / \Delta^{i-1}(\lambda)$, then $L^i(\lambda)$ is a graded $H$-lattice and $L^i(\lambda)_K \simeq L_K(\lambda_i)$ for some $\lambda_i \in I$.
\end{lem}

\begin{proof}
Fix a graded composition series 
$$
0 = \Delta^0_K(\lambda) \subset \cdots \subset \Delta^r_K(\lambda) = \Delta_K(\lambda)
$$
such that $\Delta^i_K(\lambda) / \Delta^{i-1}_K(\lambda) \simeq L_K(\lambda_i)$ for some $\lambda_i \in I$. Write $\phi : \Delta(\lambda) \ra \Delta_K(\lambda)$ for the natural map. It is an inclusion. We set 
$$
\Delta^i(\lambda) = \phi^{-1}(\Delta_K^i(\lambda)) = \Delta_K^i(\lambda) \cap \Delta(\lambda), \quad \forall \ i.
$$
Then $\Delta^i(\lambda)$ is a $H$-submodule of the $H$-lattice $\Delta(\lambda)$, hence it is a $H$-lattice. We have a $H$-morphism $\phi_i : L^i(\lambda) := \Delta^i(\lambda) / \Delta^{i-1}(\lambda) \ra \Delta^i_K(\lambda) / \Delta^{i-1}_K(\lambda) = L_K(\lambda_i)$. We claim that $L^i(\lambda)$ is a $H$-lattice. It suffice to show that it is torsion-free with respect to $k[x]$. Let $\bar{a} \in L^i(\lambda)$ and $0 \neq f(x) \in k[x]$ such that $f(x) \cdot \bar{a} = 0$. Then $f \cdot a \in \Delta^{i-1}(\lambda) = \Delta_K^{i-1}(\lambda) \cap \Delta(\lambda)$, which implies that $a \in \Delta_K^{i-1}(\lambda) \cap \Delta(\lambda) = \Delta^{i-1}(\lambda)$. Hence $\bar{a} = 0$. Since $L^i(\lambda)$ is a $H$-lattice and $L_K(\lambda_i)$ is a simple $H_K$-module, $\phi_i$ will induce an isomorphism $L^i(\lambda)_K \simeq L_K(\lambda_i)$ provided $L^i(\lambda) \neq 0$. Let $0 \neq b \in \Delta^i_K(\lambda) - \Delta^{i-1}_K(\lambda)$, then there exists some $0 \neq g(x) \in k[x]$ such that $g(x) b \in \Delta^i_K(\lambda) \cap \Delta(\lambda)$ and $g(x) b \notin \Delta^{i-1}_K(\lambda)$. Therefore $g(x) b \in \Delta^i(\lambda) - \Delta^{i-1}(\lambda)$.
\end{proof}

\subsection{}\label{sec:genericL}

By Lemma \ref{lem:compseries}, for each $\lambda \in I$ we have a graded $H$-lattice $L(\lambda) := L^r(\lambda)$ such that $\Delta(\lambda) \twoheadrightarrow L(\lambda)$. The specialization of $L(\lambda)$ to $x = \alpha \in k$ is denoted $L(\lambda)_{x = \alpha}$ to distinguish it from the simple module $L_{\alpha}(\lambda)$.

\begin{lem}\label{lem:primeann}
There exists a prime ideal $\mf{p} \lhd H$ such that $\mf{p} \cdot L(\lambda) = 0$.
\end{lem}

\begin{proof}
Let $\mathrm{Ann}_1 = \{ I \lhd H \ | \ \mathrm{GKdim} \ (\Delta(\lambda) / I \Delta(\lambda)) = 1 \}$. Since $\mathrm{GKdim} \ (\Delta(\lambda)) = 1$, the set $\mathrm{Ann}_1$ is non-empty. Since $H$ is Noetherian, we can choose $\mf{p} \in \mathrm{Ann}_1$ to be maximal with respect to inclusion. The claim is that $\mf{p}$ is prime. Assume otherwise, so that there exist ideals $I,J$ with $I J \subset \mf{p}$ but $I,J \notin \mathrm{Ann}_1$. Note that $\Delta / \mf{p} \cdot \Delta$ has GK-dimension one and is a finitely generated $H / \mf{p}$-module. Therefore $H / \mf{p}$ has GK-dimension one too. Since $\Delta / J \Delta$ has GK-dimension zero, the short exact sequence
$$
0 \ra  J \Delta / \mf{p} \Delta \ra \Delta / \mf{p} \Delta \ra \Delta / J \Delta \ra 0
$$
implies, by \cite[Proposition 8.3.11]{MR}, that $\mathrm{GKdim} (J \Delta / \mf{p} \Delta) = 1$. However, $J \Delta / \mf{p} \Delta$ is clearly a torsion $H / \mf{p}$-module. Therefore \cite[Corollary 8.3.6]{MR} implies that the GK-dimension of $J \Delta / \mf{p} \Delta$ is zero. This contradiction implies that $\mf{p}$ is prime. Let $M = \Delta / \mf{p} \cdot \Delta$. The $H$-submodule of $M$ consisting of elements that are torsion with respect to $k[x]$ is a proper submodule of GK-dimension zero. Therefore quotienting out by this submodule we may assume that $M$ is torsion free and hence free, and that $\mf{p} \cdot M = 0$. Moreover, $M_K$ is a non-zero quotient of $\Delta_K(\lambda)$ which implies that $M_K \twoheadrightarrow L_K(\lambda)$. This implies that $\mf{p} \cdot L_K(\lambda)$. Since $L(\lambda)$ is a $H$-submodule of $L_K(\lambda)$, we have $\mf{p} \cdot L(\lambda) = 0$ as required.
\end{proof}

\begin{prop}\label{lem:genericL}
Let $m_{\lambda} = \max_{\alpha \in \ell} (\dim_k L_{\alpha}(\lambda))$.
\begin{enumerate}
\item There exists a finite set $\ell_0 \subset \ell$ such that $\dim_k L_{\alpha}(\lambda) = m_{\lambda} \ \Leftrightarrow \ \alpha \notin \ell_0$.
\item The rank of $L(\lambda)$ equals $m_{\lambda}$.
\item The specialization of $L(\lambda)$ at $\alpha$ is isomorphic to the simple $H_{\alpha}$-module $L_{\alpha}(\lambda)$ for all $\alpha \in \ell - \ell_0$.
\end{enumerate}
\end{prop}

\begin{proof}
Let $\mf{p}$ be a prime ideal in $H$ such that $\mf{p} \cdot L(\lambda) = 0$. Its existence is guaranteed by Lemma \ref{lem:primeann}. Set $R := H / \mf{p}$ so that $L(\lambda)$ is an $R$-module. Since $R$ is a prime ring, the image of $k[x]$ in $R$ is a domain. Therefore the fact that $R$ is a finite $k[x]$-module of GK-dimension one implies that $R$ is a free $k[x]$-module. The map $\pi : \Spec Z(R) \ra \Spec k[x]$ is finite. Let $\mc{A}$ denote the Azumaya locus of $R$. Since $\mc{A}$ is open and dense in the irreducible variety $\Spec Z(R)$ and the map $\pi$ is finite, the set $\pi(\Spec Z(R) - \mc{A})$ is a proper closed subset of $\ell$. Therefore the set
$$
\ell_1 := \ell - \pi(\Spec Z(R) - \mc{A}) = \{ \alpha \in \ell \ | \ \pi^{-1}(\alpha) \subset \mc{A} \}
$$
is open and dense in $\ell$. Since the specialization $L(\lambda)_{x = \alpha}$ is a quotient of $\Delta(\lambda)_{x = \alpha} = \Delta_{\alpha}(\lambda)$, the simple module $L_{\alpha}(\lambda)$ is a quotient of $L(\lambda)_{x = \alpha}$. This implies that every $L_{\alpha}(\lambda)$ is an $R$-module and hence $m_{\lambda}$ is bounded above by the P.I. degree of $R$. By definition, this bounded is achieved for all $\alpha \in \ell_1$. Fix some $\alpha \in \ell_1$. Recall that $L_K(\lambda) = K \cdot L(\lambda)$. Hence
$$
\dim_K L_K(\lambda) = \rank_{k[x]} L(\lambda) \ge \dim_k L_{\alpha}(\lambda).
$$
On the other hand, since $L_{\alpha}(\lambda)$ is supported on the Azumaya locus of $R$ and $L_K(\lambda)$ is a simple module for the central localization $R_K$ of $R$, \cite[Posner's Theorem 13.6.5]{MR} together with \cite[Kaplansky's Theorem 13.3.8]{MR} imply that
$$
\dim_k(L_{\alpha}(\lambda)) = \PIdeg (R) \ge \dim_K \dim_K L_K(\lambda) = \rank_{k[x]} L(\lambda).
$$
Therefore we get the required equality.
\end{proof}

\subsection{}

We return to the case where $k$ has positive characteristic. Recall that $\mc{C}$ is the space of parameters for $\H$. We fix a parameter $\bc$. Let $\vc$ be a variable and denote by $\H_{\vc}(W)$ the rational Cherednik algebra over $k[\vc]$ such that the specialization of $\H_{\vc}(W)$ at $\vc = 1$ recovers $\H_{\bc}(W)$. Let $K = k(\vc)$ be the field of fractions of $k[\vc]$. Given $\lambda \in \Irr_k(W)$, let $\bar{\Delta}_{\vc}(\lambda) = \rH_{\vc}(W) \o_{k[V] \rtimes W} \lambda$ be the corresponding baby Verma module. It is a free $k[\vc]$-module of finite rank. Let $\bar{\Delta}_K(\lambda) = \rH_K(W) \o_{K[V] \rtimes W} K \o_k \lambda$ be the corresponding baby Verma module for $\rH_K(W)$.

\begin{lem}\label{lem:checkconditions}
\begin{enumerate}
\item  Every simple module $\lambda \in \Irr_{K}(W)$ is absolutely irreducible.
\item The module $\bar{\Delta}_K(\lambda)$ has a unique simple quotient $L_K(\lambda)$.
\item We have a natural identification of $\rH_K(W)$-modules $K \o_{k[\vc]} \bar{\Delta}_{\vc}(\lambda) = \bar{\Delta}_K(\lambda)$.
\end{enumerate}
\end{lem}

\begin{proof}
\begin{enumerate}
\item The $k$-algebra $kW$ is already split semi-simple since $k$ is assumed to be algebraically closed and the characteristic of $k$ does not divide the order of $|W|$. Therefore the $K$-algebra $K W = K \o_k k W$ is also split semi-simple.  

\item Let $\mu \in \Irr_K(W)$, considered as a $K[V] \rtimes W$-module such that $K[V]_+$ acts as zero. Adjunction implies that
$$
\Hom_{K[V] \rtimes W}(\bar{\Delta}_K(\lambda),\mu) = \Hom_{W}(K \lambda,\mu) = K \delta_{\lambda,\mu}
$$
since $K \lambda$ and $\mu$ are absolutely irreducible. Therefore, if $\phi : \bar{\Delta}_K(\lambda) \ra S$ is any simple quotient of $\bar{\Delta}_K(\lambda)$, we must have a $K[V] \rtimes W$-surjection $S \ra K \lambda$. Hence $S = \rH_K(W) \cdot K \lambda$ is the unique graded quotient of $\bar{\Delta}_K(\lambda)$.
\item The space $K \o_k \lambda \subset K \o_{k[\vc]} \bar{\Delta}_{\vc}(\lambda)$ is a $K[V] \rtimes W$ such that $K[V]_+$ acts as zero. Therefore we have a surjective map $\bar{\Delta}_K(\lambda) \ra K \o_{k[\vc]} \bar{\Delta}_{\vc}(\lambda)$ (since $K \o_k \lambda$ is a generating set for $K \o_{k[\vc]} \bar{\Delta}_{\vc}(\lambda)$). Since both modules have the same dimension over $K$, this surjection must be an isomorphism.
\end{enumerate}
\end{proof}

\subsection{A classification}

Let $\ev : K(W)[t,t^{-1}] \ra \Z[t,t^{-1}]$ be the map sending $[\lambda]$ to $\dim \lambda$. Let
\begin{equation}\label{eq:definego}
I(t) := \ev(\ch_{t,W}(V_0(\mathbf{1}))) = \left( \frac{1 - t^p}{1 - t} \right)^n.
\end{equation}

\begin{lem}\label{lem:chardiv}
If $\bc \in \mc{C}$ such that $\dim_k L_{\bc}(\lambda) = |W|p^n$, then the Poincar\'e polynomial of $L_{\bc}(\lambda)$ is given by
\begin{equation}\label{eq:poincare}
P(L_{\bc}(\lambda),t) = \frac{ \dim (\lambda) \cdot t^{pb_{\lambda^*}} P(k[V^*]^{co W},t^p) \cdot I(t)}{f_{\lambda^*}(t^p)}.
\end{equation}
\end{lem}

\begin{proof}
Let $L(\lambda)$ be the $\H_{\vc}(W)$-module described in \ref{sec:genericL}. Since it is free as a $k[\vc]$-module, the graded $W$-module structure of all specializations of $L(\lambda)$ are the same. Therefore it suffices to prove that the Poincar\'e polynomial of the specialization $L(\lambda)_{\vc = 0}$ has the desired form. We fix a filtration $\bar{\Delta}_{\vc}^i(\lambda)$ of $\bar{\Delta}_{\vc}(\lambda)$ as in Lemma \ref{lem:compseries}. Specializing to $\vc = 0$ gives a filtration $\bar{\Delta}_{0}^i(\lambda)$ of $\bar{\Delta}_{0}(\lambda)$ such that $\bar{\Delta}_0^r(\lambda) / \bar{\Delta}_0^{r-1}(\lambda) \simeq L(\lambda)_{\vc = 0}$. All the $\dd(V) \rtimes W$-modules $\bar{\Delta}_{0}^i(\lambda)$ have $p$-curvature zero. Therefore Proposition \ref{prop:Cartier} says that $\bar{\Delta}_{0}^i(\lambda) = \Ind_{A \rtimes W}^{\dd(V) \rtimes W} \bar{\Delta}_{0}^i(\lambda)^{(1)}$ for some graded $k[V^{(1)}] \rtimes W$-module $\bar{\Delta}_{0}^i(\lambda)^{(1)}$ and we have 
$$
\bar{\Delta}_{0}^r(\lambda)^{(1)} / \bar{\Delta}_{0}^{r-1}(\lambda)^{(1)} := L(\lambda)^{(1)}_{\vc = 0}, \quad \textrm{ with } \quad L(\lambda)_{\vc = 0} = \Ind_{A \rtimes W}^{\dd(V) \rtimes W} L(\lambda)^{(1)}_{\vc = 0}.
$$
This implies that $L_{\bc}(\lambda) \simeq V_0(\mathbf{1}) \o L_0(\lambda)^{(1)}$ as graded $W$-modules. Let $\rH_{\vc}(W)\text{-}\mathsf{latt}$ denote the category of $\rH_{\vc}(W)$-lattices i.e. the category of finitely generated, graded $\rH_{\vc}(W)$-modules that are free over $k[\vc]$. It is an exact, extension closed subcategory of $\grLmod{\rH_{\vc}(W)}$. Therefore it makes sense to consider the Grothendieck group $K_0(\rH_{\vc}(W)\text{-}\mathsf{latt})$ of $\rH_{\vc}(W)\text{-}\mathsf{latt}$. By our assumption on $\dim_k L_{\bc}(\lambda)$, the support of $L_{\bc}(\lambda)$ as a $Z_{\bc}(W)$-module is contained in the Azumaya locus of $\H_{\bc}(W)$. Therefore M\"uller's Theorem, \cite[Proposition 2.7]{Ramifications}, implies that each graded composition factor of the indecomposable module $\bar{\Delta}_{\bc}(\lambda)$ is of the form $L_{\bc}(\lambda)[m_i]$ for some integer $m_i$. Hence Lemma \ref{lem:compseries} says that each $L_{\vc}^i(\lambda)$ specializes at $\vc = 1$ to some $L_{\bc}(\lambda)[m_i]$. Thus, in $K_0(\rH_{\vc}(W)\text{-}\mathsf{latt})$ we have an equality $[ \bar{\Delta}_{\vc}(\lambda)] = h(t) [L_{\vc}(\lambda)]$ for some Laurent polynomial $h(t)$. This implies that $[\bar{\Delta}_{0}(\lambda)^{(1)}] = h(t) [L(\lambda)^{(1)}_{\vc = 0}]$ in the Grothendieck group of graded $k[V^{(1)}] \rtimes W$-modules. By Lemma \ref{lem:tensorreg}, the fact that $L(\lambda)_{\vc = 0}$ is isomorphic to $p^n$ copies of the regular representation implies that $L(\lambda)^{(1)}_{\vc = 0}$ is isomorphic to a graded copy of the regular representation. On the other hand, Lemma \ref{lem:pcochar} and Proposition \ref{prop:Cartier} imply that $DR^{\perp}(\bar{\Delta}_{0}(\lambda)) \simeq k[V^{(1)}]^{\mathrm{co} W} \o \lambda$, hence
$$
[k[V^{(1)}]^{\mathrm{co} W} \o \lambda] = h(t) [L(\lambda)^{(1)}_{\vc = 0}].
$$
The proof of \cite[Lemma 3.3]{Singular} implies that $h(t) = t^{- pb_{\lambda^*}} f_{\lambda^*}(t^p)$. From this one can deduce formula (\ref{eq:poincare}).
\end{proof}

Now we are finally in a position to prove Theorem \ref{thm:classify}.

\begin{proof}[Proof of Theorem \ref{thm:classify}]
Assume that $W$ and $\bc$ are chosen so that the centre of $\H_{\bc}(W)$ is smooth. Then every simple $\H_{\bc}(W)$-module has dimension $|W| p^n$. This implies that the graded character of $L_{\bc}(\lambda)$ is given by the formula of Lemma \ref{lem:chardiv}. Since $L_{\bc}(\lambda)$ is finite dimensional, the rational function on the right hand side of (\ref{eq:poincare}) must be a Laurent polynomial. This means that $f_{\lambda^*}(t^p)$ divides $P(k[V^*]^{co W},t^p) I(t)$ in $\C[t,t^{-1}]$. The formula (\ref{eq:definego}) for $I(t)$ shows that every root of $I$ is a primitive $p^{th}$ root of unity. Therefore, if $f_{\lambda^*}(t^p)$ and $I(t)$ have a non-trivial common factor, we must have $f_{\lambda^*}(\zeta^p) = 0$ for some primitive $p^{th}$ root of unity $\zeta$. But $f_{\lambda^*}(1) = \dim \lambda^* \neq 0$. Hence $f_{\lambda^*}(t^p)$ must divide $P(k[V^*]^{co W},t^p)$ in $\C[t^p,t^{-p}]$. It was shown in \cite{Singular} that if $W$ is not isomorphic to $G(m,1,n)$ or $G_4$ then one can always find some $\lambda$ for which $f_{\lambda^*}(t^p)$ does not divide $P(k[V^*]^{co W},t^p)$. Since $\bc$ has played no part in this argument, we conclude that the centre of $\H_{\bc}(W)$ is never smooth in these cases.

Conversely, it follows from Corollary \ref{cor:GMoneNsmooth} that the centre of $\H_{\bc}(G(m,1,n))$ is regular for generic $\bc$. Now we consider the group $G_4$; we have $p \neq 2,3$ in this case. By Corollary \ref{cor:equivsmooth}, it suffices to show that the blocks of $\rH_{\bc'}(W')$ are singletons for generic parameters $\bc$ and all parabolic subgroups $W'$ of $G_4$. We begin by showing that this is true for $W' = G_4$. Recall from \ref{sec:defnCherednik} the central element $\mathbf{h}^p - \mathbf{h}$ of $\H_{\bc}(G_4)$. For the remainder of the proof we follow the notation of \cite[\S 4]{Singular}. The function $\bc$ is defined by $\bc(s_i) = c_1$ and $\bc(t_i) = c_2$ for some $c_1, c_2 \in k$. To show that the block partition of $\Irr G_4$ is trivial, it suffices to show that the scalars by which the element $\mathbf{h}^p - \mathbf{h}$ acts the $L(\lambda)$'s are pairwise distinct. Let $z_1 = s_1 + \ds + s_4$ and $z_2 = t_1 + \ds + t_4 \in Z(G_4)$. The scalar $\mu_i$ by which the central element $z_i$ acts on the simple $G_4$-module $\mu$ is given by
$$
\begin{array}{c|cc}
\mu & z_1 & z_2 \\
\hline
T & 4 & 4 \\
V_1 & 4 \omega^2 & 4 \omega \\
V_2 & 4 \omega & 4 \omega^2 \\
W & -2 & -2 \\
\mathfrak{h} & -2\omega^2 & -2\omega \\
\mathfrak{h}^* & -2\omega & -2\omega^2 \\
U & 0 & 0
\end{array}
$$
We have $1 - \lambda_{s_i} = 1 - \omega^2$ and $1 - \lambda_{t_i} = 1 - \omega$. Define $d_i = \frac{- c_i}{1 - \lambda_i} \in k$. Then $\mathbf{h}^p - \mathbf{h}$ acts on $L(\mu)$ by $d_1^p \mu_1^p - d_1 \mu_1 + d_2^p \mu_2^p - d_2 \mu_2$, hence it will act by the same scalar on $L(\mu)$ and $L(\rho)$ if and only if
\begin{equation}\label{eq:someeq}
d_1 (\mu_1 - \rho_1) + d_2 (\mu_2 - \rho_2) \in \mathbb{F}_p.
\end{equation}
Since $|k| = \infty$ and the list of values $\mu_i - \rho_i$ is finite, we can always choose $d_1,d_2 \in k$ such that (\ref{eq:someeq}) does not hold, provided there is no pair $\mu \neq \rho$ such that $\mu_1 = \rho_1$ and $\mu_2 = \rho_2$. This can be checked directly: e.g. $-2 = 4$ if and only if $6 = 0$; $-2 = 4 \omega$ if and only if $p - 1 = 2 \omega$ which implies that $(p-1)^3 = p-1 = 8$ and hence $p = 9$ - both clearly contradictions.

Up to conjugacy, there is only one proper parabolic subgroup of $G_4$, it is $\Z_3$. We may assume that $\Z_3 = \langle s_1 \rangle = \{ 1, s_1, t_1 \}$. Repeating the above argument for $\Z_3$ and noting that $\Irr \Z_3 = \{ T |_{\Z_3}, V_1 |_{\Z_3}, V_2 |_{\Z_3} \}$, one sees that the block partition of $\Irr \Z_3$ will not consist of singletons precisely if equation (\ref{eq:someeq}) is satisfied for $\rho, \mu \in \{ T, V_1, V_2 \}$. 
\end{proof}


\subsection{The Kac-Weisfeiler conjecture}

The following is a result in the spirit of the Kac-Weisfeiler conjecture.

\begin{prop}
The subset $\mc{C}^{\mathrm{reg}}$ of $\mc{C}$ consisting of all parameters $\bc$ such that the dimension of every simple $\H_{\bc}(W)$-module is divisible by $p^n$ is open and dense.
\end{prop}

Note that, since the P.I. degree of $\H_{\bc}(W)$ is $p^n |W|$, $p^n$ is also the largest power of $p$ dividing the dimension of any simple $\H_{\bc}(W)$-module. Also, one can show that $\mc{C}^{\mathrm{reg}}$ is \tit{always} a proper subset of $\mc{C}$. 

\begin{proof}
First, we show that the subset $\mc{C}^{\mathrm{reg}}(W)$ of $\mc{C}$ consisting of all $\bc$ such that the dimension of every simple $\rH_{\bc}(W)$-module is divisible by $p^n$ is open and dense in $\mc{C}$. Consider the $k[\mc{C}]$-algebra $\rH_{\vc}(W)$. It is free as a $k[\mc{C}]$-module. Therefore it is a continuous family of $\mc{L}$-algebras in the sense of \cite[Definition 2.2]{PremetSkry} (for $\mc{L}$ trivial). Then \cite[Lemma 2.3]{PremetSkry} says that, for each integer $d$, the set of points $\bc \in \mc{C}$ such that $\rH_{\bc}(W)$ contains a two-sided ideal of dimension $d$ is closed in $\mc{C}$. The annihilator of a simple $\rH_{\bc}(W)$-module $M$ has codimension $(\dim_k M)^2$ (recall that $k$ is assumed to be algebraically closed). Therefore the set of all points in $\mc{C}$ for which there exists an ideal in $\rH_{\bc}(W)$ whose codimension belongs to $\{ (p^{i} r)^2 \ | \ 0 \le i < n, \ 1 \le r \le |W| \}$ is closed and its compliment is $\mc{C}^{\mathrm{reg}}(W)$. We just need to show that this set is non-empty. Take any line $\ell$ in $\mc{C}$. Then Lemma \ref{lem:genericL} implies that $\ell \cap \mc{C}^{\mathrm{reg}}(W)$ is dense in $\ell$. Note that this actually shows that there is no linear subspace of $\mc{C}$ contained in $\mc{C} - \mc{C}^{\mathrm{reg}}(W)$.

Now we treat the general case. Let $W'$ be a parabolic subgroup of $W$ and $\mc{C}_{W'}$ the space of parameters for $\H_{\bc}(W')$. Restriction of parameters defines a linear map $\rho : \mc{C} \rightarrow \mc{C}_{W'}$, its image is a non-zero linear subspace of $\mc{C}_{W'}$. We have shown that the set of points in $\mc{C}_{W'}$ for which all simple $\rH_{\bc}(W')$-module have ``maximum $p$-dimension'' is open and dense. If $\mc{C}^{\mathrm{reg}}(W')$ denotes the pre-image of this set under $\rho$ then it is open and dense in $\mc{C}$. Then it follows from Propositions \ref{prop:completesmall} and \ref{prop:assocgraded} together with Corollary \ref{cor:equivsmooth} that
$$
\mc{C}^{\mathrm{reg}} = \bigcap_{W'} \mc{C}^{\mathrm{reg}}(W'),
$$
where the intersection is over all conjugacy classes of parabolic subgroups of $W$.
\end{proof}

\begin{rem}
We remark that our reduction method (combining Propositions \ref{prop:completesmall} and \ref{prop:assocgraded} together with Corollary \ref{cor:equivsmooth}) also gives a different proof of the result \cite[Corollary 4.2]{TikaradzeKacWeisfeiler}: If $M$ is a simple $\H_{\bc}(W)$ and $b \in V$ such that the support of $M$ equals the image of $b$ in $V / W$ then $p^{|W / W_b|}$ divides $\dim_k M$.
\end{rem}

\begin{example}
When $W = S_{n+1}$, $\mc{C}^{\mathrm{reg}} = k - \mathbb{F}_p^{\times}$. One can show this as follows: by remark \ref{rem:symsing}, the algebra $\H_{\bc}(S_{n+1})$ is Azumaya for all $\bc \in k - \mathbb{F}_p$ and hence $k - \mathbb{F}_p \subseteq \mc{C}^{\mathrm{reg}}$. On the other hand, direct calculations show that $\mc{C} - \mc{C}^{\mathrm{reg}}(S_2) = \mathbb{F}_p^{\times}$. Therefore Propositions \ref{prop:completesmall} implies that $\mathbb{F}_p^{\times} \subset \mc{C} - \mc{C}^{\mathrm{reg}}$. Finally, when $\bc = 0$, $\H_{\bc}(S_{n+1}) = \dd(V) \rtimes S_{n+1}$ and, even though this is not an Azumaya algebra, it follows from Lemma \ref{prop:classifysimples} that $p^{n}$ does divide $\dim_k L$ for all simple $\dd(V) \rtimes S_{n+1}$-modules $L$. 
\end{example}  

\small{
\bibliography{biblo}{}
\bibliographystyle{alphanum} }

\end{document}